\documentclass[a4,11pt]{article}
\usepackage[latin1]{inputenc}
\usepackage{amsmath,amssymb}
\usepackage{bm}
\usepackage{hyperref} 
\usepackage{color} 
\usepackage{amsthm}

\def \dsp {\displaystyle}

\newcommand{\N}{\mathbb{N}}
\newcommand{\C}{\mathbb{C}}
\newcommand{\D}{\mathbb{D}}
\newcommand{\R}{\mathbb{R}}
\newcommand{\T}{\mathbb{T}}

\renewcommand{\Re}{\mathop{\rm Re}}
\renewcommand{\Im}{\mathop{\rm Im}}
\newcommand{\const}{\mathop{\rm const}}
\newcommand{\spn}{\mathop{\rm span}}

\newtheorem{theo}{Theorem}[section]
\newtheorem{lema}[theo]{Lemma}
\newtheorem{coro}[theo]{Corollary}
\newtheorem{prop}[theo]{Proposition}
\theoremstyle{definition}
\newtheorem{definition}{Definition}[section]
\newtheorem{remark}{Remark}[section]
\newtheorem{example}{Example}[section]

\DeclareMathOperator{\supp}{supp}

\def \dsp {\displaystyle}

\def\downbar#1{
\setbox10=\hbox{$#1$}
            \dimen10=\ht10 \advance\dimen10 by 2.5pt
            \ifdim \dimen10<15pt 
               \advance\dimen10 by -0.5pt
               \dimen11=\dimen10
               \advance\dimen10 by 2.5pt
               \lower \dimen11
            \else \lower \ht10 \fi
            \hbox {\hskip 1.5pt \vrule height \dimen10 depth \dp10}}
\def\upbar#1{
\setbox10=\hbox{$#1$}
            \dimen10=\ht10 \advance\dimen10 by \dp10 \advance\dimen10 by 2.5pt
            \ifdim \dimen10<15pt 
               \advance\dimen10 by 2pt \fi
            \raise 2.5pt \hbox {\hskip -1.5pt \vrule height \dimen10}}

\begin{document}

\title{Christoffel formula for kernel polynomials on the unit circle}
\author
{
 {C.F. Bracciali$^{a}$, A. Mart\'{i}nez-Finkelshtein$^{b}$, A. Sri Ranga$^{a}$\thanks{ranga@ibilce.unesp.br (corresponding author)}  and D.O. Veronese$^{c}$} \\[1ex]
  {\small $^{a}$DMAp, IBILCE, UNESP - Universidade Estadual Paulista,} \\
  {\small 15054-000, S\~{a}o Jos\'{e} do Rio Preto, SP, Brazil.}\\[1ex]
  {\small $^{a}$Departamento de Matem\'{a}ticas, Universidad de Almer\'{i}a, 04120 Almer\'{i}a,}\\
{\small  and Instituto Carlos I de F\'{\i}sica Te\'orica and Computacional, Granada University, Spain} \\[1ex]
  {\small $^{c}$ICTE, UFTM - Universidade Federal do Tri\^{a}ngulo Mineiro,} \\
  {\small 38064-200 Uberaba, MG, Brazil.}\\[1ex]
}


\maketitle

\thispagestyle{empty}

\begin{abstract}
  Given a nontrivial positive measure $\mu$ on the unit circle, the associated Christoffel-Darboux kernels  are  $K_n(z, w;\mu) = \sum_{k=0}^{n}\overline{\varphi_{k}(w;\mu)}\,\varphi_{k}(z;\mu)$, $n \geq 0$, where $\varphi_{k}(\cdot; \mu)$ are the orthonormal polynomials with respect to the measure  $\mu$. Let the positive measure $\nu$ on the unit circle be given by $d \nu(z) = |G_{2m}(z)|\, d \mu(z)$, where $G_{2m}$ is a conjugate reciprocal polynomial of exact degree $2m$. We establish a determinantal formula expressing $\{K_n(z,w;\nu)\}_{n \geq 0}$ directly in terms of  $\{K_n(z,w;\mu)\}_{n \geq 0}$. 
  
  Furthermore, we consider the special case of $w=1$; it is known that appropriately normalized polynomials $K_n(z,1;\mu) $ satisfy a recurrence relation whose coefficients are given in terms of two sets of real parameters  $ \{c_n(\mu)\}_{n=1}^{\infty}$ and $ \{g_{n}(\mu)\}_{n=1}^{\infty}$, with $0<g_n<1 $ for $n\geq 1$. The double sequence $\{(c_n(\mu), g_{n}(\mu))\}_{n=1}^{\infty}$ characterizes the measure $\mu$. A natural question about the relation between the parameters $c_n(\mu)$, $g_n(\mu)$, associated with $\mu$, and the sequences $c_n(\nu)$,  $g_n(\nu)$, corresponding to $\nu$, is also addressed.
  
  
  Finally, examples are considered, such as  the Geronimus weight (a measure supported on an arc of $\T$), a class of measures given by basic hypergeometric functions,  and a class of measures with hypergeometric orthogonal polynomials.
\end{abstract}

{\noindent}Keywords: Orthogonal functions, Christoffel formulas, three term recurrence relation,   orthogonal polynomials on the unit circle.

%

\section{Introduction }\label{sec:intro}
\setcounter{equation}{0}

Given a nontrivial positive measure $\mu$ on the unit circle $\T: = \{\zeta=e^{i\theta}\!\!: \, 0 \leq \theta \leq 2\pi \}$ the associated orthonormal polynomials $\varphi_n(z; \mu)=\kappa_n z^n + \text{lower degree terms}$ are defined by  $\kappa_n = \kappa_{n}(\mu) >0$ and 
\begin{equation*} 
    \int_{\T}  \overline{\varphi_{m}(\zeta; \mu)}\, \varphi_{n}(\zeta; \mu)\, d \mu(\zeta) = \int_{0}^{2\pi} \overline{\varphi_{m}(e^{i\theta}; \mu)}\, \varphi_{n}(e^{i\theta}; \mu)\, d \mu(e^{i\theta}) = \delta_{m,n}, 
\end{equation*}
for $\quad m,n=0,1,2, \dots $, where $\delta_{m,n}$ stands for the Kronecker delta. These are orthogonal polynomials on the unit circle, or in short, OPUC. A recent complete treatise on OPUC is the monograph \cite{Simon-Book-p1}. Among their fundamental properties is that all their zeros belong to the open unit disk $\D:= \{ z\in \C: \, |z|<1\} $. 

The \textit{reproducing  kernels} $K_{n}(z,\,w; \mu)$ (also known as \textit{Christoffel--Darboux} kernels or simply \textit{CD kernels}) associated with the measure $\mu$ are given by   
\begin{equation}\label{defCD}
   K_{n}(z,w; \mu) = \sum_{j=0}^{n} \overline{\varphi_j(w;\mu)}\,\varphi_j(z;\mu), \quad n \geq 0.
\end{equation}
They have been the subject of study in many recent contributions including the review \cite{Simon-Survey-2008}  on their use in the spectral theory of orthogonal polynomials and random matrices.

In what follows we use the standard notation for the reversed (or conjugate-reciprocal) polynomials:  if $q$ is an algebraic polynomial of degree $n$, then 
\[
q^{\ast}(z) := z^{n} \overline{q(1/\overline{z})}. 
\]
With this notation, the  well-known Christoffel-Darboux formula says that for $z \neq w$, 
\begin{equation}\label{CDformula}
  \begin{split}
    K_{n}(z,w; \mu)  & \dsp = \frac{ \overline{\varphi_{n+1}(w; \mu)}\,\varphi_{n+1}(z; \mu)- \overline{\varphi_{n+1}^{\ast}(w; \mu)}\,\varphi_{n+1}^{\ast}(z; \mu)} {\overline{w}\, z-1} \\[1ex]
    & \dsp = \frac{ \overline{w} z \overline{\varphi_{n}(w; \mu)}\,\varphi_{n}(z; \mu)- \overline{\varphi_{n}^{\ast}(w; \mu)}\,\varphi_{n}^{\ast}(z; \mu) } {\overline{w}\, z-1}\,,  \quad n \geq 0.
  \end{split}
\end{equation}
Notice that $K_{n}(z,0; \mu)=\overline{\varphi_{n}^{\ast}(0; \mu)}\,\varphi_{n}^{\ast}(z; \mu)$. On the other hand, if  $w \in\T$, then all zeros of $K_{n}(z,w; \mu)$ (as a polynomial in $z$) lie on $\T$, and up to a normalization factor, $(\overline{w}z-1) K_n(z,w; \mu)$ is a so-called para-orthogonal polynomial of degree $n+1$. For information concerning para-orthogonal polynomials we refer to \cite{BracRangaSwami-2016} and references therein. 
 
A multiplication of the given measure $\mu$ by a factor that is positive on its support, $\supp(\mu)$, yields a new measure and a corresponding set of OPUC and of CD kernels. It is a natural question to ask whether there is an explicit connection between these two sets.

In this paper we are interested in the case when the factor is of the form  $|g|^2$, where $g$ is a polynomial. For the orthogonality on the real axis, this is the content of the so-called Christoffel formula (see, for example, \cite{Szego-Book}), which was extended in \cite{Ismail-Ruedemann-JAT1992} to cover OPUC (see also \cite{Li-Marcellan-CATCF1999}, which generalizes \cite{Ismail-Ruedemann-JAT1992} and constitutes a nice survey of related results obtained prior to 1999, as well as some recent related results in \cite{AMT1,AMT2}). In these cases there is a determinantal expression for the ``new'' orthogonal polynomials in terms of those orthogonal with respect to $\mu$.

One of the goals of this paper is to obtain such a determinantal formula for the CD kernels on $\T$. Observe that this kind of expressions is not a trivial consequence of the analogous formulas for OPUC.

Recall that the classical Fej\'er--Riesz theorem (see \cite[\S 1.12]{Grenander-Szego-Book}) says that every non-negative trigonometric polynomial $f(\theta)$ can be written as $|g(z)|^2$, $z=e^{i\theta}$, where $g$ is an algebraic polynomial non-vanishing in $\D$. Equivalently, we can say that $f(\theta)$ is of the form $z^{-m} G(z)$, where $G$ is a \textit{self-reciprocal} polynomial (i.e., $G^*=G$) of degree $2m$.

Motivated by this result, we slightly weaken the assumptions of Fej\'er and Riesz and require the trigonometric multiplication factor of $\mu$ to be non-negative only within the support of $\mu$. More precisely, let $G_{2m}$ be a self-reciprocal polynomial of exact degree $2m$, $m \in \N$, and non-negative on $\supp(\mu)$, and let
\begin{equation}\label{defMutilde}
d \nu(\zeta) = \frac{G_{2m}(\zeta)}{\zeta^m} d \mu(\zeta) = |G_{2m}(\zeta)| d \mu(\zeta), \quad \zeta \in \T,
\end{equation}
which is also a positive measure on $\T$. 

We denote by 
\[   
     z_1, z_2, \dots, z_{2m}
\]
the zeros of $G_{2m}$.  Those of them not on $\T$ must appear in pairs symmetric with respect to $\T$; notice that no $z_j$ is $=0$.  However, unlike in the case of Fej\'er--Riesz, if $\supp(\mu)\neq \T$, zeros of $G_{2m}$ on $\T$ also can be simple, as long as the hypothesis of positivity of $G_{2m}(\zeta)/\zeta^{m}$ on $\supp(\mu)$ (i.e., the positivity of $\nu$ on $\supp(\mu)$) is preserved.\footnote{For instance, if $\supp(\mu)=\{e^{i\theta}:\, 0<\theta_1\leq \theta \leq \theta_2 <2\pi \}$, we can consider the self reciprocal polynomial $G_2(z) = e^{-i\alpha/2} e^{-i\beta/2} (z -e^{i \alpha}) (z -e^{i\beta})$,  with $0 < \alpha \leq \theta_1$ and $\theta_2 \leq \beta < 2\pi$. Then the rational function
\[
    \frac{G_{2}(\zeta)}{\zeta} = \frac{G_{2}(e^{i\theta})}{e^{i\theta}} = 4 \sin \frac{\theta-\alpha}{2} \sin \frac{\beta - \theta}{2}
\]
is positive on $\supp(\mu)$, but not on the entire $\T$; see Example~\ref{example43} in Section~\ref{sec:examples} below. }

In what follows we will be mainly interested in the case when all $z_j$'s are pairwise distinct (or equivalently, when all zeros of  $G_{2m}$ are simple). 

\begin{definition} \label{def:admissible}
Given $m\in \N$, we call a set $\mathcal P=\{p_0, p_1, \dots, p_{2m}\}$ of $2m+1$ not identically $0$ algebraic polynomials $p_j$ \textit{admissible} if 
$p_0(z)\equiv 1$, $p_{2m}(z)= z^m$, 
\begin{equation}
\label{def-pj}
p_{j}(z)=\sum_{j=\max\{ 0, j-m\}}^{\min \{j,m\}} b_{ij}\, z^i, \quad j= 1, \dots, 2m-1,
\end{equation}
and  either one of the following three conditions is satisfied:
\begin{equation} \label{assumption1}
\deg p_j < j \quad \text{for } j=1, \dots, 2m,
\end{equation}
\begin{equation} \label{assumption3}
\deg p_j < m \quad \text{for } j= 1, \dots, 2m-1,
\end{equation}
or
\begin{equation} \label{assumption2}
p_j(0)=0 \quad \text{for } j=1, \dots, 2m.
\end{equation}
\end{definition}

\begin{theo}\label{thm1}
For an admissible set $\mathcal P=\{p_0, p_1, \dots, p_{2m}\}$ and $w\in \C$ define 
%
\begin{equation}\label{def-Qj}
Q_j(z,w) := p_j(z) K_{n+2m-j}(z,w; \mu), \quad j=0, 1, \dots, 2m,
\end{equation}
and the $(2m+1)\times (2m+1)$ matrix
\begin{equation}\label{defQmatrix}
\bm Q(z,w)=\bm Q^{(\mathcal P)}(z,w) :=\begin{pmatrix}
Q_0(z,w) & Q_1(z,w) & \dots & Q_{2m}(z,w) \\
Q_0(z_1,w) & Q_1(z_1,w) & \dots & Q_{2m}(z_1,w) \\
\vdots & \vdots &  \ddots & \vdots \\
Q_0(z_{2m},w) & Q_1(z_{2m},w) & \dots & Q_{2m}(z_{2m},w)
\end{pmatrix}.
\end{equation}
Then, there exists a polynomial $C_n(w) =C^{(\mathcal P)}_n(w)$ of degree $\leq (2m+1)(n+m)$ such that
\begin{equation}\label{mainformula}
\det \bm Q(z,w)=  C_n(\overline{w}) \, G_{2m}(z)K_n(z,w; \nu) .
\end{equation}
\end{theo}

Observe that for certain values of $w$, both sides of \eqref{mainformula} can vanish, in which case the identity in \eqref{mainformula} is formally correct, but practically useless. Thus, a natural question is about sufficient conditions for $C_n\neq 0$.

\begin{theo}\label{thm1bis}
	Let all the zeros of $G_{2m}$ be simple. 
	For an admissible set $\mathcal P=\{p_0, p_1, \dots, p_{2m}\}$ and $w\in \C$, with the notations of Theorem~\ref{thm1}, if either 
	\begin{enumerate}
		\item[i)] $|w|\geq 1$, condition \eqref{assumption1} holds and polynomials $Q_1(\cdot, w), Q_2(\cdot, w), \dots,   Q_{2m}(\cdot, w)$ are linearly independent, or
		\item[ii)] $|w|\leq 1$, condition \eqref{assumption3} holds and polynomials $Q_0(\cdot, w), Q_1(\cdot, w), \dots,   Q_{2m-1}(\cdot, w)$  are linearly independent, or
		\item[iii)] $0<|w|\leq 1$, condition \eqref{assumption2} holds   and polynomials $Q_1(\cdot, 1/\overline{w})$, $Q_2(\cdot, 1/\overline{w})$, \dots,   $Q_{2m}(\cdot, 1/\overline{w})$  are linearly independent,
	\end{enumerate}
	then $C_n^{(\mathcal P)}(\overline{w}) \neq 0$.
\end{theo}

\begin{remark} \label{Rmk-1.1} If the polynomial $G_{2m}$ has non-simple zeros, then the results above still hold if one replaces the polynomials in each row of the matrix $\bm Q$ by the respective derivatives in accordance with the order of multiplicity.  For example, if $z_1 \neq z_2  = z_3$,  then the fourth row of $\bm Q$ must be replaced by 
	\[ Q^{\prime}_{0}(z_2, w),  Q^{\prime}_{1}(z_2, w),  \cdots, Q^{\prime}_{2m-1}(z_2, w),  Q^{\prime}_{2m}(z_2, w),
	\]
	where $Q_j'$ stands for its derivative with respect to $z$.
\end{remark}

Given an admissible set $\mathcal P=\{p_0, p_1, \dots, p_{2m}\}$, we define $\widehat{\mathcal P}=\{\widehat{p}_0, \widehat{p}_1, \dots, \widehat{p}_{2m}\}$, with
\begin{equation}\label{definitionHat}
\widehat{p}_j(z)  := z^{j} \overline{p_j(1/\overline{z})}= z^{j-\deg(p_j)} p_j^*(z), \quad j=0, 1, \dots, 2m.
\end{equation}
Observe that $\widehat{(\widehat{\mathcal P})}= \mathcal P$.

\begin{prop} \label{prop1}
	A set of polynomials $\mathcal P=\{p_0, p_1, \dots, p_{2m}\}$ is admissible if and only if $\widehat{\mathcal P}=\{\widehat{p}_0, \widehat{p}_1, \dots, \widehat{p}_{2m}\}$  is. Moreover, $\mathcal P$ satisfies \eqref{assumption1} (resp., \eqref{assumption2}), then $\widehat{\mathcal P}$ satisfies \eqref{assumption2} (resp., \eqref{assumption1}). 
	
	Additionally, if $w\neq 0$, 
	\begin{equation}\label{equivalence}
	C^{(\mathcal P)}_n(w) \neq 0 \quad \Leftrightarrow \quad C^{(\widehat{\mathcal P})}_n\left( \frac{1}{\overline{w}}   \right)\neq 0.
	\end{equation}
\end{prop}

It would be nice to have a simple recipe for constructing an admissible set $\mathcal P$ for which \eqref{mainformula} renders a non-trivial identity for the CD kernel $K_n(z,w; \nu)$. Obviously, there is no ``universal'' $\mathcal P$ such that $=C^{(\mathcal P)}_n(w)$ in \eqref{mainformula} is $\neq 0$ for \emph{all} $w\in \C$. However, there is a simple admissible set that guarantees this, at least for $|w|=1$.

It is easy to check that $\mathcal P=\{p_0, p_1, \dots, p_{2m}\}$, with
\begin{equation}\label{example1}
p_j(z)=z^{\lfloor j/2 \rfloor}=\begin{cases}
z^{j/2}, & \text{if $j$ is even},\\
z^{(j-1)/2} , & \text{if $j$ is odd},
\end{cases} \quad j=0, 1, \dots, 2m,
\end{equation}
is admissible and satisfies both \eqref{assumption1} and \eqref{assumption3}. The corresponding $\widehat{\mathcal P}=\{\widehat{p}_0, \widehat{p}_1, \dots, \widehat{p}_{2m}\}$ is
\begin{equation}\label{example2}
\widehat p_j(z)=z^{\lfloor j+1/2\rfloor}=\begin{cases}
z^{j/2}, & \text{if $j$ is even,}  \\
z^{(j+1)/2} , & \text{if $j$ is odd,}
\end{cases} \quad j=0, 1, \dots, 2m,
\end{equation}
satisfying, by Proposition~\ref{prop1}, condition  \eqref{assumption2}.

\begin{prop}\label{proponT}
Let the admissible set of polynomials $\mathcal P=\{p_0, p_1, \dots, p_{2m}\}$ be given by \eqref{example1}, and let $w$ be such that
\begin{equation}\label{conditiononw}
    K_n(0,w;\mu)\neq 0 \quad \text{and} \quad \deg K_n(\cdot, w;\mu)=n \quad \text{for } n\in \N.
\end{equation}
Then $Q_0(\cdot, w)$, $Q_1(\cdot, w)$,  \dots,  $ Q_{2m}(\cdot, w)$ are linearly independent.  In particular, it holds for   $|w|=1$.  
\end{prop}

\begin{coro}
For the admissible sets of polynomials $\mathcal P$ and $\widehat{\mathcal P}$, given by \eqref{example1} and \eqref{example2}, respectively, both $ C^{(\mathcal P)}_n(w) \neq 0$ and $ C^{(\widehat{\mathcal P})}_n(w) \neq 0$ in \eqref{mainformula} when $|w|=1$.
\end{coro}
 

\begin{example}\label{exampleLeb}
Let us consider the normalized Lebesgue measure on $\T$, 
\begin{equation}\label{normalizedLebesgue}
d\mu(\zeta) = \frac{1}{2\pi} |d\zeta|,  \quad \zeta\in \T.
\end{equation}
Then
$$
K_n(z,w;\mu)=\frac{1-\overline{w}^{n+1} z^{n+1}}{1-\overline{w} z},  \quad n\geq 0,
$$
so that all $w\neq 0$ satisfy conditions \eqref{conditiononw} from Proposition~\ref{proponT}. In particular, for all such $w$, and for the admissible set $\mathcal P$ given by \eqref{example1},  $C_n^{(\mathcal P)}(\overline{w}) \neq 0$ in \eqref{mainformula}. However, $K_n(z,0;\mu)\equiv 1$, which implies that polynomials $Q_j(z,0)$ in \eqref{def-Qj} are linearly dependent for any choice of the admissible set of polynomials $\mathcal P$, and hence, one cannot find an admissible set for which $\det \bm Q(\cdot ,0) \not \equiv 0$. Clearly, we still can recover $K_{n}(z,0; \nu)=\overline{\varphi_{n}^{\ast}(0; \nu)}\,\varphi_{n}^{\ast}(z; \nu)$ by taking limit,
$$
 G_{2m}(z) K_{n}(z,0; \nu) =\lim_{w\to 0}  \frac{1}{C_n^{(\mathcal P)}(\overline{w})}\det \bm Q(z,w).
$$
\end{example}

The proofs of the assertions above are gathered in Section~\ref{sec:proofsKernels}.

In Section~\ref{Sec-ModifiedCDKernel} we consider an interesting particular case of $w=1$, for which $K_{n}(\cdot ,1;\mu)$ constitute an instance of paraorthogonal polynomials on $\T$. A convenient ``symmetrization'' of these polynomials was found in \cite{Costa-Felix-Ranga-JAT2013}; it was shown there that the appropriately normalized $K_{n}(\cdot , 1; \mu)$, that we denote by $R_n(\cdot;\mu)$ (see the precise definition in Section~\ref{Sec-ModifiedCDKernel}) satisfy a three term recurrence relation of the form 
\begin{equation}  \label{Eq-TTRR-Rn}
 \begin{array}{l}
R_{n+1}(z;\mu) = \left[(1+ic_{n+1})z + (1-ic_{n+1})\right] R_{n}(z;\mu) \\[1ex]
   \hspace{40ex} - \ 4(1-g_{n})g_{n+1} z R_{n-1}(z;\mu), 
 \end{array}
\end{equation}
for $n \geq 0$, with $R_{-1}(z) = 0$ and $R_{0}(z) = 1$.  Sequences   $\{c_n\}_{n=1}^{\infty}= \{c_n(\mu)\}_{n=1}^{\infty}$ and $\{g_{n}\}_{n=1}^{\infty}= \{g_{n}(\mu)\}_{n=1}^{\infty}$ are both real, with $0<g_n<1 $ for $n\geq 1$. As shown in \cite{BracRangaSwami-2016,Castillo-Costa-Ranga-Veronese-JAT2014,Costa-Felix-Ranga-JAT2013}, the double sequence $\{(c_n(\mu), g_{n}(\mu))\}_{n=1}^{\infty}$ is a parametrization of the measure $\mu$, alternative to its Verblunsky coefficients. Thus, a natural question is   the relation between the parameters $c_n(\mu)$, $g_n(\mu)$, associated with $\mu$, and the sequences $c_n(\nu)$,  $g_n(\nu)$, corresponding to $\nu$. These questions will be addressed in Section~\ref{Sec-ModifiedCDKernel}. Since the statement of the corresponding results requires introducing a considerable piece of notation, we postpone it to the aforementioned section.

Finally, in Section~\ref{sec:examples} we consider four different applications of our formulas: a rather straightforward case when $\mu$ is the Lebesgue measure on $\T$, the Geronimus weight (a measure supported on an arc of $\T$),  a class of measures given by basic hypergeometric functions, and a class of measures with hypergeometric OPUC.


\section{Proof of the Christoffel formula for kernels}\label{sec:proofsKernels}
\setcounter{equation}{0}

First we discuss a characterization of the kernel polynomials $K_n$.

Let $w\in \C$ be fixed. With the positive measure $\mu$ on $\T$ we consider the complex-valued measure on $\T$ given by, 
$$
d\mu_w(\zeta) =  (1 -\overline{w} \zeta) d \mu(\zeta).
$$
The following simple lemma will be useful in the forthcoming proofs:
\begin{lema}\label{lemma:positive}
	Let  $f$ be an integrable function on $\T$ such that either one of the following condition is satisfied:
    \begin{enumerate}
\item[i)] $ |w|\leq 1$ and 
\begin{equation}\label{nullidentity}
	 \int  \left|f(\zeta) \right|  d \mu_w(\zeta) = 0,
	\end{equation}
    or
 \item[ii)] $ |w|\geq 1$ and 
\begin{equation}\label{nullidentity2}
	 \int \overline{\zeta} \left|f(\zeta) \right|  d \mu_w(\zeta) = 0.
	\end{equation} 
\end{enumerate}
	Then $f=0$ $\mu$-a.e.~(in case when $|w|\neq  1$) and $\mu\big|_{\T\setminus\{w\}}$-a.e., otherwise.
	
	In particular, if $f$ is a polynomial and $\mu$ has an infinite number of points of increase, then i) or ii) imply that $f\equiv 0$.
\end{lema}
\begin{proof}
Consider i) first. For $w=0$ the statement is trivial, so let $0<|w|\leq 1$. By assumptions of the lemma,
$$
	\int  \left|f(\zeta) \right| \left(\frac{1}{\overline{w}}-\zeta\right) d \mu(\zeta) = 0,
$$
and hence,
$$
\int  \left|f(\zeta) \right| \left(\frac{1}{|w|}-e^{i\theta}\zeta \right) d \mu(\zeta) = 0,\quad \theta=-\arg (w).
$$
In particular, taking the real part, we get
$$
\int  \left|f(\zeta) \right|\Re \left(\frac{1}{|w|}-e^{i\theta} \zeta \right) d \mu(\zeta) = 0,
$$
and it remains to notice that
$$
\Re \left(\frac{1}{|w|}-e^{i\theta} \zeta \right) >0,
$$
unless $|w|=1$ and $\zeta =w$.

In the case ii), we have that
$$
\int  \left|f(\zeta) \right| \left(\overline{w} -\overline{\zeta} \right) d \mu(\zeta) =\overline{\int   \left|f(\zeta) \right| \left( w-\zeta \right) d \mu(\zeta)} = 0,
$$
so that
$$
\int   \left|f(\zeta) \right| \left( |w|-e^{i\theta}\zeta \right) d \mu(\zeta)=0,
$$
and again,
$$
\Re \left(|w|-e^{i\theta} \zeta \right) >0,
$$
unless $|w|=1$ and $\zeta =w$.
\end{proof}

\begin{lema}\label{lemma:uniqueness}
	For a fixed $w\in \C$ and $n \in \N$, the CD kernel $K_n(z,w;\mu)$ is a polynomial in $z$ of degree $\leq n$, characterized up to a constant factor by the following orthogonality relations,
	\begin{equation} \label{Eq-Rn-Ortogonality}
	\int  \overline{\zeta^{s}}   K_{n}(\zeta,w;\mu )  d \mu_w(\zeta) = 0,\quad  1 \leq s \leq n,
	\end{equation} 
	and the additional condition
	\begin{enumerate}
		\item[a)] if $|w|\geq 1$, then $K_n(z,w;\mu )$ is of degree exactly $n$;
		\item[b)] if $|w|\leq 1$, then
$$
\int     K_n(\zeta,w;\mu )  d \mu_w(\zeta) \neq 0.
$$		
	\end{enumerate}  
\end{lema}
\begin{proof}
	Orthogonality conditions \eqref{Eq-Rn-Ortogonality} are a straightforward consequence of \eqref{CDformula} and of the well-known relations for the reversed polynomials,
	$$
	\int  \overline{\zeta^{s}}   \varphi_n^*(\zeta;\mu )  d \mu(\zeta) = 0,\quad  1 \leq s \leq n, \quad \text{and} \quad \int    \varphi_n^*(\zeta;\mu )  d \mu(\zeta) \neq 0.
	$$
	Since for $|w|\geq 1$, $\varphi_n(w;\mu )\neq 0$, using the definition \eqref{defCD} of $K_n(z,w;\mu )$ we conclude that it is a polynomial in $z$ of degree $=n$. For $|w|\leq 1$, using \eqref{CDformula} and the fact that $ \varphi_{n+1}^*$ does not vanish inside or on the unit disk, we see that
	\begin{equation*}
		\begin{split}
		\int    K_n(\zeta,w;\mu )  d \mu_w(\zeta) &=\int    K_n(\zeta,w;\mu )   (1-\overline{w} \zeta) d \mu(\zeta)\\
		& =\overline{\varphi_{n+1}^{\ast}(w;\mu )}\,\int    \varphi_{n+1}^*(\zeta;\mu )  d \mu(\zeta) \neq 0.
		\end{split}
	\end{equation*}

	Let us prove that these relations characterize the CD kernel.
	
	Assume that $|w|\geq 1$. If $P$ is a polynomial of degree exactly $n$, satisfying 
	\begin{equation} \label{Eq-Rn-Oalt}
	\int  \overline{\zeta^{s}}   P(\zeta)  d \mu_w(\zeta) = 0,\quad  1 \leq s \leq n,
	\end{equation} 
	then there exists a constant $c \neq 0$ such that  $L(z):= c K_n(z,w; \mu)- P(z)$ is of degree $\leq n-1$. By hypothesis, $L$ satisfies the same orthogonality conditions, so that
	\begin{align*}
	0 = & \int  \overline{\zeta L(\zeta)}   L(\zeta)  d \mu_w(\zeta) =\int  \overline{\zeta}  |L(\zeta)|^2 \mu_w(\zeta),
\end{align*}
and it remains to apply Lemma~\ref{lemma:positive}, ii), to conclude that $L\equiv 0$. 


On the other hand, if  $P$ of degree $\leq n$ satisfies \eqref{Eq-Rn-Oalt} and is such that 
$$
\int    P(\zeta)  d \mu_w(\zeta) \neq 0,
$$
then there exists a non-zero constant, let us denote it by $c$ again, such that
$$
\int   ( P(\zeta)- c\, K_n(\zeta,w; \mu))  d \mu_w(\zeta) = 0.
$$
Combining it with \eqref{Eq-Rn-Ortogonality} and \eqref{Eq-Rn-Oalt} we get that for $s=0, 1, \dots, n$,
\begin{equation*}
\begin{split}
\int   \overline{\zeta^{s}} ( P(\zeta)- c\, K_n(\zeta,w; \mu))  d \mu_w(\zeta) =\int   \overline{\zeta^{s}} ( P(\zeta)- c\, K_n(\zeta,w; \mu)) (1-\overline{w} \zeta) d \mu(\zeta) = 0. 
\end{split}
\end{equation*}
It means that
$$
( P(\zeta)- c\, K_n(\zeta,w; \mu))(1-\overline{w} \zeta)=\const \varphi_{n+1}(\zeta; \mu).
$$
Since $\varphi_{n+1}$ cannot vanish at $1/\overline{w}$, we conclude that $P(\zeta)\equiv c\, K_n(\zeta,w; \mu)$.
\end{proof}

\medskip

In order to prove  Theorem~\ref{thm1} we need some preparatory steps.

With the notation of Section~\ref{sec:intro}, it is immediate to check that  $\det \bm Q(z,w)$ is an algebraic polynomial in $z$ (of degree $\leq (n+2m)$) and  in $\overline{w}$ (of degree $\leq (m + n)( 2 m+1) $). Furthermore, for each $w\in \C$ it  vanishes at the zeros $z_1, \dots, z_{2m}$ of $G_{2m}$. Thus, we can write
\begin{equation}\label{definitionofA}
\det \bm Q(z,w) =G_{2m}(z)A_n(z,w) ,
\end{equation}
where $A_n$ is an algebraic polynomial in $z$ (of degree $\leq n$) and in $\overline{w}$.
We need to show that for each $w\in \C$, there exists a constant $C$ such that
\begin{equation}\label{identityforA}
A_n(z,w) =C  \,  K_n(z,w; \nu).
\end{equation}
If this is established, the polynomial dependence of $C$ from $\overline{w}$ (as well as on the admissible set $\mathcal P$ chosen, see Definition~\ref{def:admissible}) is a straightforward consequence of \eqref{definitionofA}--\eqref{identityforA}.
 
We prove \eqref{identityforA} by appealing to the characterization of $K_n$ given in Lemma~\ref{lemma:uniqueness}. 

By \eqref{Eq-Rn-Ortogonality} and the definition of $\nu$, kernels $K_n(z, w; \nu)$ satisfy
\begin{equation*}
\int  \overline{\zeta^{s}}\,  G_{2m}(\zeta) K_{n}(\zeta,w;\nu )  d \mu_w(\zeta) = 0, \quad  m+1 \leq s \leq m+n.
\end{equation*}
Thus, a \textit{necessary} condition for \eqref{identityforA} is that 
\begin{equation*}
\int  \overline{\zeta^{s}}\,  G_{2m}(\zeta) A_n(\zeta)  d \mu_w(\zeta) = 0, \quad  m+1 \leq s \leq m+n.
\end{equation*}
This is always true, and it is an immediate consequence of the following lemma:
\begin{lema}\label{lemma11}
	For $j=0, 1, \dots, 2m$,
	\begin{equation}
	\label{termsBis}
	\int  \overline{\zeta^{s}}  Q_j(\zeta, w) d \mu_w(\zeta) =0,   \quad s= m+1, \dots, m+n.
	\end{equation}
	Thus, 
	\begin{equation}
	\label{OrthDet}
	\int  \overline{\zeta^{s}}  \det \bm Q(\zeta, w)  d \mu_w(\zeta) = 0, \quad  m+1 \leq s \leq m+n.
	\end{equation}
	\end{lema}
	\begin{proof}
In order to calculate
\begin{equation}
\label{terms}
		\int  \overline{\zeta^{s}}\,  p_j(\zeta) K_{n+2m-j}(\zeta,w; \mu)  d \mu_w(\zeta)    \quad \text{for } m+1 \leq s \leq m+n, \quad 0\leq j \leq 2m,
		\end{equation}
with account of \eqref{def-pj}, it is sufficient to find the values of
		\begin{equation}
		\label{TildaOrtogonalityMod}
		\int  \overline{\zeta^{s-r}} \,   K_{n+2m-j}(\zeta,w; \mu)  d \mu_w(\zeta)   \quad  \text{for } m+1 \leq s \leq m+n,  \quad 0\leq j \leq 2m,
		\end{equation}
for $  \max\{ 0, j-m\}\leq r \leq \min \{j,m\}$. 
		
Since $   r \leq \min \{j,m\}\leq m $, we get
		$$
		s-r\geq m+1-r\geq 1.
		$$
		Analogously,  from $   r \geq \max\{ 0, j-m\}\geq j-m $, we conclude that
		$$
		s-r\leq m+n-r\leq n+2m-j.
		$$
		By \eqref{Eq-Rn-Ortogonality} it follows that all integrals in \eqref{TildaOrtogonalityMod} (and consequently, in \eqref{terms}) vanish, which yields \eqref{termsBis}--\eqref{OrthDet}.
		\end{proof}

So, the necessary condition (orthogonality) always holds. Now we go for a sufficient condition, given by a) and b) of Lemma~\ref{lemma:uniqueness}. 

Recall first the following well known fact, that we state just as a remark.

\begin{remark}\label{remarkzeros}
	By the maximum principle, $|\varphi^* _n(z; \mu)|>|\varphi_n(z; \mu)|$ for $|z|<1$,  and $|\varphi^* _n(z; \mu)|<|\varphi_n(z; \mu)|$ for $|z|>1$. By \eqref{CDformula}, 
	$$
	K_{n}(z,w; \mu)=0 \quad \Leftrightarrow \quad \overline{\varphi_{n+1}(w; \mu)}\,\varphi_{n+1}(z; \mu)- \overline{\varphi_{n+1}^{\ast}(w; \mu)}\,\varphi_{n+1}^{\ast}(z; \mu)=0,
	$$
	which for $|w|>1$ can be rewritten as
	$$
	\frac{\varphi_{n+1}(z; \mu)}{\varphi_{n+1}^{\ast}(z; \mu)}= \overline{\left(\frac{\varphi_{n+1}^{\ast}(w; \mu) }{\varphi_{n+1}(w; \mu)}\right)}.
	$$
	The right hand side is of absolute value $<1$, so that equality is possible only for $|z|<1$. Same analysis is valid for the other case and we conclude that the zeros of $K_{n}(z,w; \mu)$ (as a polynomial of $z$) are of absolute value $>1$ (respectively, $=1$ or $<1$) if $|w|<1$ (respectively, $|w|=1$ or $|w|<1$).
\end{remark}	

\begin{lema}\label{lemma3.2}
	Let $A_n$ be defined by \eqref{definitionofA} and $w\in \C$ fixed. 	The following conditions are necessary and sufficient for $A_n(\cdot,w)\equiv 0$:
	\begin{enumerate}
		\item[i)]  $|w|\geq 1$ and
		\begin{equation}\label{equivalence1}
		\deg (\det \bm Q(\cdot,w)) <n+2m; 
		\end{equation}
		\item[ii)]  $|w|\leq 1$ and
		\begin{equation}\label{equivalence4}
		\int \overline{\zeta^m}\det \bm Q(\zeta, w) \, d\mu_w(\zeta)= 0. 
		\end{equation}
	\end{enumerate}
\end{lema}
\begin{proof}
	$A_n(\cdot,w)\equiv 0$ is clearly a sufficient condition in \emph{i)} and \emph{ii)} for  \eqref{equivalence1} and \eqref{equivalence4}, respectively. So, we prove that this is also necessary.
	
	For \textit{i)}, if $\deg A_n(\cdot,w)< n$ (which is equivalent to $\deg (\det \bm Q(\cdot,w)) <n+2m$), then by \eqref{OrthDet},
	$$
	0 = \int  \overline{\zeta^{m+1} A_n(\zeta, w)}  \det \bm Q(\zeta,w )  d \mu_w(\zeta) = \int \overline{\zeta} |A_n(\zeta,w)|^2 |G_{2m}(\zeta)|  d \mu_w(\zeta),
	$$
	and we use again Lemma~\ref{lemma:positive}, ii), to conclude that $A_n(\cdot,w)\equiv 0$.
		
	For \textit{ii)}, if 
	$$
	\int \overline{\zeta^m}\det \bm Q(\zeta,w) \, d\mu_w(\zeta)=\int A_n(\zeta,w) |G_{2m}(\zeta)|\, d\mu_w(\zeta)=0,
	$$
	then by \eqref{OrthDet} we have in fact that for $s=0, 1, \dots, n$, 
	$$
	\int \overline{\zeta^s} A_n(\zeta, w) |G_{2m}(\zeta)|\, d\mu_w(\zeta)=\int \overline{\zeta^s} A_n(\zeta, w)(1-\overline{w} \zeta) |G_{2m}(\zeta)|\, d\mu(\zeta)= 0.
	$$
	Since $\deg A_n(\cdot, w)\leq n$, we conclude that
	$$
	A_n(\zeta,w)(1-\overline{w} \zeta) = c\, \varphi_{n+1}(\zeta;\nu),
	$$
	But $\varphi_{n+1}(\zeta;\nu)$ cannot vanish at $\zeta=1/\overline{w}\notin \D$, which implies that $c=0$, and $A_n(\cdot, w)\equiv 0$.
	\end{proof}

A combination of Lemmas~\ref{lemma11} and \ref{lemma3.2} constitutes the proof of Theorem~\ref{thm1}. Indeed, let $|w|\geq 1$. Consider the identity \eqref{definitionofA}; by Lemma~\ref{lemma3.2}, i), either $A_n(\cdot, w) \equiv 0 $ or $\deg A_n(\cdot, w)=n$. In the latter case by the characterization in Lemma~\ref{lemma:uniqueness}, a), 
$$
A_n(z,w) = C K_{n}(\zeta,w;\nu ) .
$$
On the other hand, if $|w|\leq 1$, then again by Lemma~\ref{lemma3.2}, ii), either $A_n \equiv 0 $ or
$$
\int A_n(\zeta) |G_{2m}(\zeta)|\, d\mu_w(\zeta)\neq 0,
$$
in which case by Lemma~\ref{lemma:uniqueness}, b), $A_n$ coincides, up to a constant factor, with $K_{n}(\cdot,w;\nu )$.

Now we turn to Theorem~\ref{thm1bis}.

Checking \eqref{equivalence1} or \eqref{equivalence4} is not straightforward. Seeking a more explicit algebraic condition,  we introduce a notation for the minors of the matrix $\bm Q$: the one, obtained by deleting its first row and column,
\begin{equation} \label{defDelta}
\Delta_0(w) = \det \begin{pmatrix}
  Q_1(z_1, w) & \dots & Q_{2m}(z_1, w) \\
\vdots &\ddots &\vdots  \\
 Q_1(z_{2m}, w) & \dots & Q_{2m}(z_{2m}, w)
\end{pmatrix},
\end{equation}
and the minor obtained from $\bm Q$ by deleting its first row and its last column,
\begin{equation} \label{defDeltam}
\Delta_m(w) = \det \begin{pmatrix}
Q_0(z_1, w) & \dots & Q_{2m-1}(z_1, w) \\
\vdots &\ddots &\vdots  \\
Q_0(z_{2m}, w) & \dots & Q_{2m-1}(z_{2m}, w)
\end{pmatrix}.
\end{equation}
		
\begin{lema}\label{lemma3.2bis}
Let $A_n$ be defined by \eqref{definitionofA} and $w\in \C$ fixed, and let either one of the following conditions hold:
\begin{itemize}
	\item \eqref{assumption1} with $|w|\geq 1$; 
	\item \eqref{assumption2} with $|w|\leq 1$.
\end{itemize}  
Then
$$
A_n(\cdot,w)\equiv 0  \quad \Leftrightarrow\quad \Delta_0(w)= 0.
$$

On the other hand, if condition \eqref{assumption3} holds with $|w|\leq 1$, then 
$$
A_n(\cdot,w) \equiv 0 \quad \Leftrightarrow\quad \Delta_m(w)= 0.
$$

\end{lema}
\begin{proof}
Under assumptions of the first part, observe that
	\begin{equation}\label{linearcomb}
		G_{2m}(z)A_n(z,w) = \det \bm Q(z,w) = \Delta_m^{(n)}(w) K_{n+2m}(z,w) + \text{span $\{Q_1, \dots, Q_{2m}\}$}.
	\end{equation}

	We prove the $(\Rightarrow)$ part first, assuming  $\Delta_0(w)\neq 0$. 
	If \eqref{assumption1}, then 
	the leading coefficient of $\det \bm Q(\cdot, w)$ is $\Delta_m^{(n)}(w) \times$ leading coefficient of $K_{n+2m}(\cdot, w)\neq 0$, so $A_n(\cdot, w)\not \equiv 0$.
	And if \eqref{assumption2} takes place, then (see Remark~\ref{remarkzeros}) 
	$$
	\det \bm Q(0,w) = \Delta_0(w) K_{n+2m}(0,w)\neq 0,
	$$
	so $A_n(\cdot, w)\not \equiv 0$ again.

Now, for the $(\Leftarrow)$ part, let $\Delta_0(w)= 0$ and $n\geq 1$. Then, by \eqref{linearcomb},
\begin{equation}\label{caseDelta0}
G_{2m}(z)A_n(z,w) = \det \bm Q(z,w) = \text{span $\{Q_1, \dots, Q_{2m}\}$}.
\end{equation}

	Under assumption \eqref{assumption1} with $|w|\geq 1$, we conclude that 
	$$
	 \deg(\det \bm Q(\cdot, w))<n+2m,
	$$
	and it remains to apply the assertion \textit{i)}  of Lemma~\ref{lemma3.2}.
	
	If we have \eqref{assumption2} with $|w|\leq 1$, then by  \eqref{caseDelta0},
	$$
	G_{2m}(z)A_n(z,w) = \det \bm Q(z,w) = z G_{2m}(z)B_{n-1}(z,w),
	$$
	where $B_{n-1}(\cdot, w)$ is again a polynomial of degree $\leq n-1$. By \eqref{OrthDet},
	\begin{align*}
	0= & \int  \overline{\zeta^{m+1} B_{n-1}(\zeta,w)}  \det \bm Q(\zeta,w)  d \mu_w(\zeta)\\
	 = & \int  \overline{\zeta^{m+1} B_{n-1}(\zeta,w)} \zeta G_{2m}(\zeta)B_{n-1}(\zeta)  d \mu_w(\zeta)\\
	= & \int  |B_{n-1}(\zeta,w)|^2  |G_{2m}(\zeta)|   d \mu_w(\zeta),
	\end{align*}
	and by Lemma~\ref{lemma:positive}, \emph{i)}, we conclude again that $A_n(\cdot, w) \equiv 0$.
	
	Finally, notice that under  \eqref{assumption3}, for each $j=0, 1, \dots, 2m-1$,
	$$
	\int  \overline{\zeta^{m}}  Q_j(\zeta,w) d \mu_w(\zeta) 
	$$
	is a linear combination of integrals of the form
	$$
	\int  \overline{\zeta^{m-s}}  K_{n+2m-j}(\zeta,w) d \mu_w(\zeta)    \quad \text{with}\quad  s\leq m-1, 
	$$
	and by \eqref{Eq-Rn-Ortogonality}, each of these integrals vanishes, so that \eqref{assumption3} implies
	\begin{equation}\label{integralvanishing}
	  \int  \overline{\zeta^{m}}  Q_j(\zeta,w) d \mu_w(\zeta)  =0, \quad j=0, 1, \dots, 2m-1.
	\end{equation}
	In consequence,
	\begin{align*}
	\int \overline{\zeta^m}\det \bm Q(\zeta, w) \, d\mu_w(\zeta) & = \Delta_m(w) \int  \overline{\zeta^{m}}  Q_{2m}(\zeta, w) d \mu_w(\zeta) \\
	& = \Delta_m(w) \int    K_{n}(\zeta,w) d \mu_w(\zeta).
	\end{align*}
	By Lemma~\ref{lemma:uniqueness}, b), the integral in the right hand side does not vanish. Thus, the integral in the left hand side is $=0$ if and only if $\Delta_m(w)=0$, and it remains to use part \emph{ii)} of Lemma~\ref{lemma3.2} to conclude the proof.	
	\end{proof}

Let us look at some sufficient conditions that guarantee that either $\Delta_0$ or $\Delta_m$ do not vanish.
\begin{lema}\label{lemma25}
Let all the zeros of $G_{2m}$ be simple, and either one of the following conditions hold:
\begin{itemize}
\item \eqref{assumption1} with $|w|\geq 1$; 
\item \eqref{assumption2} with $|w|\leq 1$.
\end{itemize} 
If polynomials $Q_1(\cdot, w), \dots, Q_{2m}(\cdot, w)$ are linearly independent, then $\Delta_0(w)\neq 0$.

Furthermore, if \eqref{assumption3} holds with $|w|\leq 1$, then linear independence of the polynomials $Q_0(\cdot, w), Q_1(\cdot, w), \dots, Q_{2m-1}(\cdot, w)$ implies that $\Delta_m(w)\neq 0$.
\end{lema}
\begin{proof}
Assume that $\Delta_0(w)= 0$, which means that the columns of the matrix in the right hand side of  \eqref{defDelta}
are linearly dependent: there exist constants $c_1(w), \dots, c_{2m}(w)$,  possibly depending on $w$, \textit{not all zero}, such that 
$$
\sum_{j=1}^{2m} c_j(w) \begin{pmatrix}
Q_j (z_1,w) \\
\vdots \\
Q_j (z_{2m},w)
\end{pmatrix}=\begin{pmatrix}
0 \\
\vdots \\
0
\end{pmatrix}.
$$
In other words, the polynomial 
$$
S(z,w):= \sum_{j=1}^{2m} c_j(w) Q_j(z,w) 
$$
vanishes at the zeros of $G_{2m}$, and by the assumed linear independence of $Q_j$'s, $S(\cdot ,w)\not \equiv 0$.

With \eqref{assumption1}, and since  all the zeros of $G_{2m}$ are simple,
$$
S(z,w)=G_{2m}(z)P(z,w), \quad \deg P(\cdot, w) \leq n-1, \quad P(\cdot, w)\not \equiv 0,
$$
and if $|w|\geq 1$, then by \eqref{termsBis},
\begin{align*}
0= & \int  \overline{\zeta^{m+1} P(\zeta, w)}  S(\zeta, w)  d \mu_w(\zeta) = \int  \overline{\zeta^{m+1} P(\zeta, w)}  G_{2m}(\zeta)P(\zeta, w)  d \mu_w(\zeta)\\
= & \int \overline{\zeta} |P(\zeta, w)|^2  |G_{2m}(\zeta)|    d \mu_w(\zeta),
\end{align*}
so that again by Lemma~\ref{lemma:positive}, ii), we conclude that this is impossible.

In the same vein, with assumption \eqref{assumption2}, 
$$
S(z,w)=G_{2m}(z)z P(z, w), \quad \deg P(\cdot, w) \leq n-1, \quad P(\cdot, w)\not \equiv 0, 
$$
and since $|w|\leq 1$, by \eqref{termsBis},
\begin{align*}
0= & \int  \overline{\zeta^{m+1} P(\zeta, w)}  S(\zeta, w)  d \mu_w(\zeta) = \int  \overline{\zeta^{m+1} P(\zeta, w)} \zeta G_{2m}(\zeta)P(\zeta, w)  d \mu_w(\zeta)\\
= & \int  |P(\zeta, w)|^2  |G_{2m}(\zeta)|    d \mu_w(\zeta),
\end{align*}
and using Lemma~\ref{lemma:positive}, i), we arrive at the same conclusion.

Analogously, if $\Delta_m(w)= 0$, by the same reasoning  there exists a  polynomial 
$$
S(z, w):= \sum_{j=0}^{2m-1} c_j(w() Q_j(z, w) \not \equiv 0
$$
(by the linear independence of $Q_j$'s), that again vanishes at the zeros of $G_{2m}$, so that
$$
S(z, w)=G_{2m}(z)P(z, w), \quad \deg P(\cdot, w) \leq n. 
$$
Using that $S\in \text{span}\{Q_0, \dots, Q_{2m-1}\} $ and \eqref{integralvanishing}, we conclude that
$$
\int  \overline{\zeta^{m}}  S(\zeta, w) d \mu_w(\zeta)= \int  P(\zeta, w) |G_{2m}(\zeta)| d \mu_w(\zeta)   =0,
$$
and it remains to use Lemma~\ref{lemma:positive}, i). 
\end{proof}

Now we can finish the proof of Theorem~\ref{thm1bis}. Indeed, if $|w|\geq 1$, \eqref{assumption1} holds and $Q_1(\cdot, w), Q_2(\cdot, w), \dots,   Q_{2m}(\cdot, w)$ are all linearly independent, then we have $\Delta_0(w)\neq 0$ (Lemma~\ref{lemma25}), then $A_n(\cdot, w)\not \equiv 0$ (Lemma~\ref{lemma3.2bis}), which implies that in \eqref{mainformula}, $C_n(\overline{w})\neq 0$ (Lemma~\ref{lemma3.2}).

On the other hand, if $|w|\leq 1$, \eqref{assumption3} holds and $Q_0(\cdot, w), Q_1(\cdot, w), \dots,   Q_{2m-1}(\cdot, w)$ are all linearly independent, then we have $\Delta_m(w)\neq 0$ (Lemma~\ref{lemma25}), then $A_n(\cdot, w)\not \equiv 0$ (Lemma~\ref{lemma3.2bis}), which again implies that in \eqref{mainformula}, $C_n(\overline{w})\neq 0$ (Lemma~\ref{lemma3.2}).

Finally, the case of $0<|w|\leq 1$, with condition \eqref{assumption2} holds, $G_{2m}(0)\neq 0$, and when polynomials $Q_1(\cdot, 1/\overline{w})$, $Q_2(\cdot, 1/\overline{w})$, \dots,   $Q_{2m}(\cdot, 1/\overline{w})$  are linearly independent, follows from the first case considered and Proposition~\ref{prop1} (see its proof next).
 
\bigskip

Now we turn to Proposition~\ref{prop1}. 

\begin{proof}[Proof of Proposition~\ref{prop1}]
Recall that admissibility of $ \mathcal P$ means that for $j=0, 1,\dots , 2m$,
$$
p_j \in \spn \{z^k:\,  \max\{ 0, j-m\}\leq k \leq \min \{j,m\} \},
$$
which makes the first statement of the Proposition a straightforward consequence of \eqref{definitionHat}. 

Let us use the superscript $\mathcal P$ in the definition \eqref{def-Qj}   to indicate the dependence of $Q_j$'s on the  admissible set explicitly:
$$
Q_j^{(\mathcal P)}(z,w) = p_j(z) K_{n+2m-j}(z,w; \mu), \quad j=0, 1, \dots, 2m.
$$
Using the identity
$$
K_{n}(z,w;\mu) = \overline{w}^n z^n \overline{K_{n}\left(\frac{1}{\overline{z}},\frac{1}{\overline{w}};\mu \right)}
$$
we can rewrite it as 
\[
Q_j^{(\mathcal P)}(z,w) = p_j(z) \overline{w}^{n+2m-j} z^{n+2m-j} \overline{K_{n+2m-j}\left(\frac{1}{\overline{z}},\frac{1}{\overline{w}};\mu \right)}, 
\]
Conjugating and using the definition \eqref{definitionHat} we conclude that
\begin{equation}\label{identityforQs}
Q_j^{(\mathcal P)}\left(z,w  \right) = \overline{w}^{n+2m-j} z^{n+2m}\, \overline{Q_j^{(\widehat{\mathcal P})}\left(\frac{1}{\overline{z}},\frac{1}{\overline{w}}  \right)}.
\end{equation}
Denoting $z_0=z$, we have  by \eqref{mainformula} that
$$
C_n^{(\mathcal P)}(\overline{w}) \, G_{2m}(z)K_n(z,w; \nu) = \det \bm Q(z,w) = \det \left( Q_j^{(\mathcal P)}\left(z_i,w  \right)\right)_{i, j=0, 1, \dots, 2m}.
$$
Using \eqref{identityforQs} we continue this set of identities as
\begin{equation*}
\begin{split}
C_n^{(\mathcal P)}(\overline{w}) \,& G_{2m}(z)K_n(z,w; \nu) =\det \left(  \overline{w}^{n+2m-j} z_i^{n+2m}\, \overline{Q_j^{(\widehat{\mathcal P})}\left(\frac{1}{\overline{z_i}},\frac{1}{\overline{w}}  \right)}\right)_{i, j=0, 1, \dots, 2m}\\
& = \overline{w}^{(2 m+1) (m + n)} z^{n+2m} \prod_{i=1}^{2m} \left(z_i^{n+2m}\right)\,  \det \left(    \overline{Q_j^{(\widehat{\mathcal P})}\left(\frac{1}{\overline{z_i}},\frac{1}{\overline{w}}  \right)}\right)_{i, j=0, 1, \dots, 2m}.
\end{split}
\end{equation*}
Since by Proposition~\ref{prop1}, $\widehat{\mathcal P}$ is also admissible, we get from \eqref{mainformula}:
\begin{equation*}
\begin{split}
C_n^{(\mathcal P)}(\overline{w}) \,& G_{2m}(z)K_n(z,w; \nu) \\
& = \overline{w}^{(2 m+1) (m + n)} z^{n+2m} \prod_{i=1}^{2m} \left(z_i^{n+2m}\right)\, \overline{C_n^{(\widehat{\mathcal P})}\left(\frac{1}{w}  \right) \, G_{2m}\left(\frac{1}{\overline{z}}  \right)K_n\left(\frac{1}{\overline{z}},\frac{1}{\overline{w}} ;\nu \right)}.
\end{split}
\end{equation*}
By the definition of self-reciprocal polynomial,
$$
z^{2m} \overline{ G_{2m}\left(\frac{1}{\overline{z}}  \right)}= \, G_{2m}(z),
$$
hence it finally simplifies to
\begin{equation*}
\begin{split}
C_n^{(\mathcal P)}(\overline{w}) \,& G_{2m}(z)K_n(z,w; \nu) \\
& = \overline{w}^{(2 m+1) (m + n)} z^{n} G_{2m}(z) \prod_{i=1}^{2m} \left(z_i^{n+2m}\right)\, \overline{C_n^{(\widehat{\mathcal P})}\left(\frac{1}{w}  \right) \,  K_n\left(\frac{1}{\overline{z}},\frac{1}{\overline{w}} ;\nu \right)}.
\end{split}
\end{equation*}
This proves \eqref{equivalence}.
\end{proof}

\medskip

We finish this section by establishing that the choice of  $\mathcal P$ given in \eqref{example1} (and consequently, of $\widehat{\mathcal P} $ in \eqref{example2}) renders a non-trivial identity for the CD kernel $K_n(z,w;\nu)$ when $|w|=1$. 

\begin{proof}[Proof of Proposition~\ref{proponT}]
In order to prove our statement, we need to modify the notation, reflecting explicitly the dependence on $m$ and $n$, but not on $\mathcal P$. Hence, along this proof we denote 
$$
Q_j^{(n,m)}(z,w):= z^{\lfloor j/2\rfloor}K_{n+2m-j}(z,w; \mu), \quad j=0, 1, \dots, 2m.
$$

Observe that with the assumptions \eqref{conditiononw}, 
$$
\deg Q^{(n,m)}_0 (\cdot,w)=\deg K_{n+2m}(\cdot,w; \mu)=n+2m> \deg Q^{(n,m)}_j (\cdot,w), \quad j=1, \dots, 2m,
$$
so it is sufficient to establish the linear independence of $Q^{(n,m)}_1(\cdot, w)$,  $Q^{(n,m)}_2(\cdot, w)$, \dots,   $Q^{(n,m)}_{2m}(\cdot, w)$. We do it by induction in $m$. For $m=1$, the system 
$$
\{Q_j^{(n,1)}(z,w) \}_{j=1}^{2} = \{ K_{n+1}(z,w; \mu), z K_{n}(z,w; \mu)  \} 
$$
is linearly independent, just because, again by assumptions \eqref{conditiononw}, $Q_1^{(n,1)}(0,w)=K_{n+1}(0,w; \mu)\neq 0$, and $Q_2^{(n,1)}(0,w)= 0$.

Assuming that the linear independence is proved already for 
$\{Q_j^{(n,m-1)}(z,w) \}_{j=1}^{2(m-1)}$ for all $n\geq 0$, let  
$$
S(z,w):=\sum_{j=1}^{2m} a_j Q_j^{(n,m)}(z,w)\equiv 0.
$$
Again, $Q_1^{(n,m)}(0,w)=K_{n+2m-j}(0,w; \mu)\neq 0$, and $Q_j^{(n,m)}(0,w)=0$ for $j=2, \dots , 2m$, so that $a_1=0$, and thus,
$$
S(z,w) = \sum_{j=2}^{2m} a_j  Q_j^{(n,m)}(z,w)\equiv 0.
$$ 
But for $j=3, \dots, 2m$,
$$
\deg ( Q_2^{(n,m)}(z,w))=\deg (z K_{n-2m-2}(z,w; \mu)=n-2m-1 > \deg  ( Q_j^{(n,m)}(z,w)),  
$$
which yields now that also $a_2=0$. It remains to observe that
$$
Q_j^{(n,m)}(z,w) = z\,  Q_{j-2}^{(n-1,m-1)}(z,w), \quad j=3, \dots, 2m,
$$
which, by the induction hypothesis, are linearly independent. This yields that also $a_3=\dots =a_{2m}$. 
Hence,
$$
\left\{ Q_j^{(n,m)}(\cdot,w):\,  j=1, \dots, 2m\right\}
$$
are linearly independent. 
The proposition is proved. 
\end{proof}

\section{A three-term recurrence for CD kernels} \label{Sec-ModifiedCDKernel}
\setcounter{equation}{0}

We now give some special consideration to the case in which $w = 1$. 
It is well known that the sequence $K_n(z,1;\mu)$ satisfies a three-term recurrence relation (see \cite[Thm.\,2.1]{Costa-Felix-Ranga-JAT2013}). Moreover, with an appropriate normalization this recurrence takes an especially convenient form, which we briefly summarize here.

In what follows, we use the standard notation $\Phi_{n}(z;\mu) = \varphi_{n}(z;\mu)/\kappa_n(\mu)$, $n \geq 0$, for the monic OPUC, as well as for the Verblunsky coefficients $\alpha_{n}(\mu) = - \overline{\Phi_{n+1}(0;\mu)}$, $n \geq 0$. It is well known that  $|\alpha_n(\mu)| < 1$ for $n \geq 0$, and that the sequence  $\{\alpha_n(\mu) \}_{n\geq 0}$ uniquely determines the  measure $\mu$ on $\T$ and allows to recover the monic OPUC via the Szeg\H{o} recurrence,
\begin{equation*} 
\Phi_{n}(z;\mu) =  z \Phi_{n-1}(z;\mu) - \overline{\alpha}_{n-1}(\mu)\,      \Phi_{n-1}^{\ast}(z;\mu), \quad 
n \geq 1,
\end{equation*}
(see, for example,  \cite{ENZG1} and \cite{Simon-Book-p1}). 

Let 
\begin{equation*}
\tau_{n}(\mu) := \overline{ \left(\frac{ \varphi_{n}^{\ast}(1;\mu)}{  \varphi_{n}(1;\mu)  }\right)} =   \frac{ \Phi_{n}(1;\mu)}{   \Phi_{n}^{\ast}(1;\mu)  }\in \T;
\end{equation*}
then $\tau_n$'s also satisfy a  relation, which can be used to compute them recursively in terms of $\alpha_n$'s,
\begin{equation*}
\tau_{n}(\mu) = \frac{\tau_{n-1}(\mu) - \overline{\alpha}_{n-1}(\mu)}{1 - \tau_{n-1}(\mu)\alpha_{n-1}(\mu)},  \quad n \geq 1,
\end{equation*}
starting with $\tau_0(\mu)=1$.

We define the sequence 
\begin{equation}\label{parametricSeq1}
g_{n}(\mu) := \frac{1}{2} \frac{\big|1 - \tau_{n-1}(\mu) \alpha_{n-1}(\mu)\big|^2}{\big[1 - \Re\big(\tau_{n-1}(\mu)\alpha_{n-1}(\mu)\big)\big]}, \quad n \geq 1.
\end{equation}
It is easy to check that all $g_n\in (0,1)$, so the terms of the following sequence are all positive:
\begin{equation} \label{defXi}
\xi_n(\mu) := \xi_0(\mu) \prod_{j=1}^{n}\big(1-g_j(\mu)\big), \quad n \geq 1, \qquad \xi_0(\mu) :=   \int_{\T} d\mu(\zeta).
\end{equation}
With this notation we introduce the normalized CD kernels
\begin{equation} \label{Eq-Normalized-CDKernel}
    R_n(z;\mu) := \xi_n(\mu)\, K_{n}(z,1;\mu) , \quad n \geq 0.
\end{equation} 
It turns out (see ~\cite{Costa-Felix-Ranga-JAT2013}) that they satisfy the following three-term recurrence formulas:
\begin{equation} \label{Eq-TTRR-Rn1}
R_{n+1}(z,\mu) = [(1+ic_{n+1})z + (1-ic_{n+1})]\,R_{n}(z,\mu) - 4d_{n+1} z\,R_{n-1}(z,\mu),   
\end{equation}
for $n \geq0$, with $R_{-1}(z,\mu) = 1$ and $R_{0}(z,\mu) = 1$, where both $\{c_n\}_{n \geq 1}$ and $\{d_{n+1}\}_{n \geq 1}$ are  real sequences. In fact,
\begin{equation}\label{cSeq}
c_{n} = c_{n}(\mu)= \frac{\Im \big(\tau_{n-1}(\mu)\alpha_{n-1}(\mu)\big)} {\Re(\tau_{n-1}\alpha_{n-1})-1} \in \R, \quad n\geq 1,
\end{equation}
and
\begin{equation}\label{def:CS1}
d_{n+1} = d_{n+1}(\mu)  = \big(1-g_{n}(\mu)\big)g_{n+1}(\mu), \quad n \geq 1,
\end{equation}
with $g_n(\mu)$ from \eqref{parametricSeq1}. In the standard terminology, this means that $\{d_{n+1}(\mu)\}_{n \geq 1}$ is a positive chain sequence, and $\{g_{n+1}(\mu)\}_{n\geq 0}$  is a parameter sequence  for  $\{d_{n+1}(\mu)\}_{n \geq 1}$.  
From  \cite{BracRangaSwami-2016, Castillo-Costa-Ranga-Veronese-JAT2014, MFinkelshtein-SRanga-Veronese-2016} it is known that the double sequence   $\{(c_n, g_n)\}_{n \geq 1}$ determines uniquely the measure $\mu$  on $\T$, as it happens also to the Verblunsky coefficients. There is actually a direct connection between these two parametrizations: if from $\{\alpha_n\}_{n \geq 0}$  to $\{(c_n, g_n)\}_{n \geq 1}$ we can navigate using \eqref{parametricSeq1} and \eqref{cSeq}, the inverse mapping is given by 
\begin{equation} \label{Eq-Inverse-Formula}
\alpha_{n-1}  = \frac{1}{\tau_{n-1}} \frac{1-2g_{n}-ic_{n}}{1 - ic_{n}}, \quad \tau_{n} =  \frac{1-ic_{n}}{1+ic_{n}}\,\tau_{n-1}, \quad n \geq 1,
\end{equation}
with $\tau_0=1$.

The Christoffel transformation  $\mu \mapsto \nu$, given by \eqref{defMutilde}, induces the corresponding transformation both on the Verblunsky coefficients and on the $ (c_n, g_n)$ parametrization of measures. Formulas for $\alpha_n(\nu)$ can be derived from \cite{Ismail-Ruedemann-JAT1992} or \cite{Li-Marcellan-CATCF1999} by evaluating the OPUC at the origin. A natural question, that we address next, is the existence of an effective way of constructing $\{(c_n(\nu), g_n(\nu)\}_{n \geq 1}$ from $\{(c_n(\mu), g_n(\mu)\}_{n \geq 1}$.

If $\mathcal P=\{p_0, p_1, \dots, p_{2m}\}$ is an admissible set (see Definition~\ref{def:admissible}) satisfying \eqref{assumption1}, then by Lemma~\ref{lemma3.2bis}, $\Delta_0(1)\neq 0 $ (see \eqref{defDelta}), so that we can rewrite \eqref{mainformula} as
\begin{equation} \label{ConnectionRn}
\dsp {\gamma_{n}^{(m)}}  G_{2m}(z) R_{n}(z;\nu) =  R_{n+2m}(z;\mu) + \sum_{j=1}^{2m} a_{2j-1}^{(n,m)}  p_{j}(z) R_{n+2m-1}(z;\mu) , \quad n\ge 0,
\end{equation}
where each coefficient $a_{j}^{(n,m)}$, $j=1,2, \dots, 2m$, can be  computed in a trivial fashion as a ratio of two minors of the matrix $
\bm Q^{(\mathcal P)}(z,1)$ in \eqref{defQmatrix}, with $\Delta_0(1)$ in the denominators, times $1/\xi_{n+2m-1}(\mu)$.

Let us consider particularly the admissible set \eqref{example1}, so that \eqref{ConnectionRn} takes the form
\begin{equation} \label{Eq-Connection-Formula-Rn}
  \begin{split} 
     {\gamma_{n}^{(m)}}  G_{2m}(z) R_{n}(z;\nu) = & R_{n+2m}(z;\mu) + \sum_{j=1}^{m} a_{2j-1}^{(n,m)}   z^{j-1} R_{n+2m-2j+1}(z;\mu) \\ 
    & + \sum_{j=1}^{m} a_{2j}^{(n,m)}  z^{j} R_{n+2m-2j}(z;\mu), \quad n\ge 0.
  \end{split}
\end{equation}

\begin{theo} \label{Thm-Rn-Rn-Conection}
	Let $N = n+2m$. The coefficients $\gamma_{n}^{(m)}$ in the left hand side of the connection formula \eqref{Eq-Connection-Formula-Rn} satisfy
	\[
	\frac{\gamma_{n}^{(m)}}{\gamma_{n-1}^{(m)}} = \frac{1}{2} \left[[1-ic_{N-1}(\mu)] \frac{[1-ic_{N}(\mu)]+ a_1^{(n,m)}}{[1-ic_{N-1}(\mu)]+ a_1^{(n-1,m)}} + [1+ic_{N}(\mu)]  \right],  \quad n \geq 1, 
	\]
	with
$\gamma_{0}^{(m)} = [\overline{G_{2m}(0)}]^{-1}\prod_{j=1}^{2m}[1+ic_{j}(\mu)]$.

Furthermore, the the  coefficients  $c_n(\nu)$ and  $g_{n}(\nu)$ corresponding to the measure $\nu$ from   \eqref{defMutilde} are, for $n \geq 1$,
\[
    c_n(\nu) = \frac{i}{2} \frac{\gamma_{n-1}^{(m)}}{\gamma_{n}^{(m)}}\left[[1-ic_{N-1}(\mu)] \frac{[1-ic_{N}(\mu)]+ a_1^{(n,m)}}{[1-ic_{N-1}(\mu)]+ a_1^{(n-1,m)}} - [1+ic_{N}(\mu)]  \right], 
\]
and 
\[
     g_n(\nu) =1- \frac{\gamma_{n-1}^{(m)}}{\gamma_{n}^{(m)}}\frac{\xi_{N}(\mu) + \dsp \sum_{j=1}^{m}a_{2j-1}^{(n,m)}\xi_{N-2j+1}(\mu) +  \sum_{j=1}^{m}a_{2j}^{(n,m)}\xi_{N-2j}(\mu)} 
    {\xi_{N-1}(\mu) + \dsp \sum_{j=1}^{m}a_{2j-1}^{(n-1,m)}\xi_{N-2j}(\mu) + \sum_{j=1}^{m}a_{2j}^{(n-1,m)}\xi_{N-2j-1}(\mu)}, 
\]
with $\xi_n(\mu)$ defined in \eqref{defXi}.
\end{theo} 

\begin{proof}
It is easily seen from \eqref{Eq-TTRR-Rn1} that $R_{n}(0;\cdot) = (1-ic_n)R_{n-1}(0;\cdot)$, $n \geq 1$.  Hence, we always have 
\begin{equation} \label{Eq-R_n-Normalization}
\Re \left( \frac{R_{n}(0;\mu)}{R_{n-1}(0;\mu)} \right)   = 1 \quad \mbox{and} \quad  \Re \left( \frac{R_{n}(0;\nu)}{R_{n-1}(0;\nu)} \right)   = 1, \quad n \geq 1.
\end{equation}

One of the features of the modified kernels $R_n(.; \mu)$ and $R_n(.;\nu)$ is that they are conjugate-reciprocal polynomials, so that 
\begin{equation} \label{Eq-Connection-Formula-Rn-Recipro}
	\begin{split}
		{\overline{\gamma}_{n}^{(m)}}  G_{2m}(z) R_{n}(z;\nu) = & R_{n+2m}(z;\mu) + \sum_{j=1}^{m} \overline{a}_{2j-1}^{(n,m)}   z^{j} R_{n+2m-2j+1}(z;\mu)  \\ 
		& + \sum_{j=1}^{m} \overline{a}_{2j}^{(n,m)}  z^{j} R_{n+2m-2j}(z; \mu), \quad n\ge 0.
	\end{split}
\end{equation}

From \eqref{Eq-Connection-Formula-Rn} and \eqref{Eq-Connection-Formula-Rn-Recipro} we get
\[
     \gamma_{n}^{(m)} G_{2m}(0)R_{n}(0;\nu)  =  \big[R_{n+2m}(0;\mu) + a_{1}^{(n,m)} R_{n+2m-1}(0;\mu)\big], \quad n \geq 0,
\]
and
\[
   \overline{\gamma}_{n}^{(m)} G_{2m}(0) R_n(0;\nu) =  R_{n+2m}(0;\mu), \quad n \geq 0.
\]
Hence, 
\[
    \gamma_{0}^{(m)} = \frac{\,\overline{R_{2m}(0;\mu)}\,}{\,\overline{G_{2m}(0)}\,} = \frac{R_{2m}(0;\mu) + a_{1}^{(0,m)} R_{2m-1}(0;\mu)}{G_{2m}(0)}, 
\]
which also establishes the value of $\gamma_{0}^{(m)}$ as stated in the theorem. 

Next, by \eqref{Eq-R_n-Normalization}, 
\[ 
   \begin{array}{ll} 
    2 & \dsp = \frac{R_{n}(0;\nu)}{R_{n-1}(0;\nu)} + \frac{\overline{R_{n}(0;\nu)}}{\overline{R_{n-1}(0;\nu)}} \\[3ex]
     & \dsp = \frac{\gamma_{n-1}^{(m)}}{\gamma_{n}^{(m)}} \left[\frac{R_{n+2m}(0;\mu) + a_1^{(n,m)}R_{n+2m-1}(0;\mu)}{R_{n+2m-1}(0;\mu) + a_1^{(n-1,m)}R_{n+2m-2}(0;\mu)} + \frac{\overline{R_{n+2m}(0;\mu)}}{\overline{R_{n+2m-1}(0;\mu)}} \right], \quad n \geq 1.
   \end{array}
\]
From this and from the identity $R_n(0;\mu) = [1-ic_n(\mu)]R_{n-1}(0;\mu)$ we easily obtain the results for $\gamma_{n}^{(m)}$. 
Likewise, using that $-c_n(\nu) = \Im\big(R_n(0;\nu)/R_{n-1}(0;\nu) \big)$ we find the expression for $c_n(\nu)$.

From the reproducing properties of the kernels $K_n(z;\mu)$ and $K_n(z;\nu)$, we have 
\begin{equation*} \label{Eq-Rn-ReproducingProperty}
   \int_{\T} p(z) \overline{R_n(z;\mu)} d \mu(z) = \xi_{n}(\mu)\, p(1) \quad \mbox{and} \quad \int_{\T} p(z) \overline{R_n(z;\nu)} d \nu(z) = \xi_{n}(\nu)\, p(1),
\end{equation*}
when $p(z)$ is a polynomial of degree at most $n$. Thus, from \eqref{Eq-Connection-Formula-Rn-Recipro} we have 
\[
    \gamma_{n}^{(m)}\,\xi_n(\nu) = \xi_{n+2m} + \sum_{j=1}^{m}a_{2j-1}^{(n,m)}\xi_{n+2m-2j+1}(\mu) + \sum_{j=1}^{m}a_{2j}^{(n,m)}\xi_{n+2m-2j}(\mu).  
\]
This leads to the result for $g_n(\nu)$ as stated. 
\end{proof}

Let us consider the simplest non-trivial case, when $m=1$, so that  $d \nu(z) =  z^{-1} G_{2}(z) \, d \mu(z)$, where $G_2(z) = az^2 + 2bz + \overline{a}$ is such that $G_2(\zeta)/\zeta$ is positive on $\supp(\mu)$. Let $\{K_n(\cdot, 1; \mu)\}_{n\geq 0}$ and $\{K_n(\cdot, 1;\nu)\}_{n\geq 0}$ be respectively the CD kernel polynomials with respect to $\mu$ and $\nu$. If the zeros $z_1$ and $z_2$ of $G_2$ are distinct then 
 	\[
 	u^{(n,1)} G_{2}(z) K_{n}(z, 1;\nu) =  K_{n+2}(z, 1;\mu) + v_{1}^{(n,1)} K_{n+1}(z, 1;\mu) 
 	+  v_{2}^{(n,1)} zK_{n}(z, 1;\mu), \quad n \geq 0,
 	\]
 	where
 	\[
 	v_{1}^{(n,1)} = \frac{1}{\Delta_{1}^{(n)}} \big[ K_{n+2}(z_2, 1;\mu) z_1 K_{n}(z_1, 1;\mu) - K_{n+2}(z_1, 1;\mu)z_2K_n(z_2, 1;\mu)\big]
 	\]
 	and
 	\[
 	v_{2}^{(n,1)}  = \frac{1}{\Delta_{1}^{(n)}} \big[ K_{n+2}(z_2, 1;\mu)  K_{n+1}(z_1, 1;\mu) - K_{n+2}(z_1, 1;\mu)K_{n+1}(z_2, 1;\mu)\big].
 	\]
 	Here,
 	\[
 	\Delta_{1}^{(n)} 
 	= K_{n+1}(z_1, 1;\mu) z_2 K_{n}(z_2, 1;\mu) - K_{n+1}(z_2, 1;\mu)  z_1 K_{n}(z_1, 1;\mu). 
 	\]

 In the case when  $G_2$ has a multiple zero then the formulas for $v_{1}^{(n,1)}$ and $v_{2}^{(n,1)}$ can be replaced by  
\begin{align*}
 	v_{1}^{(n,1)} = &  \frac{-1}{\Delta_{1}^{(n)}} \big[K_{n+2}(z_1, 1;\mu)z_1 K^{\prime}_n(z_1, 1;\mu) \\ 
 	& \phantom{11111}- K^{\prime}_{n+2}(z_1, 1;\mu) z_1 K_{n}(z_1, 1;\mu) + K_{n+2}(z_1, 1;\mu)K_n(z_1, 1;\mu)\big],
 	\\
 	v_{2}^{(n,1)}  = & \frac{-1}{\Delta_{1}^{(n)}} \big[K^{\prime}_{n+2}(z_1, 1;\mu)K_{n+1}(z_1, 1;\mu) - K_{n+2}(z_1, 1;\mu)  K^{\prime}_{n+1}(z_1, 1;\mu)\big].
\end{align*}
Here, 
\begin{align*}
 \Delta_{1}^{(n)} 
 	= & K_{n+1}(z_1, 1;\mu) z_1 K^{\prime}_{n}(z_1, 1;\mu) - K^{\prime}_{n+1}(z_1, 1;\mu)  z_1 K_{n}(z_1, 1;\mu) \\
 	 & + K_{n+1}(z_1, 1;\mu) K_{n}(z_1, 1;\mu). 
\end{align*}

In terms of the normalized CD kernels $R_n(z;\mu) = \xi_n(\mu) K_{n}(z,1;\mu)$ and $R_n(z;\nu) = \xi_n(\nu) K_{n}(z,1;\nu)$, the relation \eqref{Eq-Connection-Formula-Rn} takes the form 
\begin{equation} \label{Eq-Connection-Formula-Rn-1}
  \begin{array}{l}
    \dsp \gamma_{n}  G_{2}(z) R_{n}(z;\nu) =  R_{n+2}(z;\mu) +  a_{1}^{(n)}   R_{n+1}(z;\mu)  +  a_{2}^{(n)}  z R_{n}(z;\mu), 
  \end{array}
\end{equation}
for $n \geq 0$, where 
 $(a_1^{(n)},  a_{2}^{(n)}) $ is the solution of the system of equations
\[
    \left[  \begin{array}{cc}
                 R_{n+1}(z_1;\mu) & z_1R_{n}(z_1;\mu) \\[1ex]
                 R_{n+1}(z_{2};\mu) &  z_2R_{n}(z_{2};\mu)   
              \end{array} \right]  
     \left[  \begin{array}{c}
                  a_1^{(n,1)} \\[1ex]
                  a_2^{(n,1)} 
               \end{array} \right]  =      
     \left[  \begin{array}{c}
                -R_{n+2}(z_1;\mu) \\[1ex]
                -R_{n+2}(z_2;\mu)  
         \end{array} \right] ,                
\]
assuming $z_1 \neq z_2$. 
A direct consequence of Theorem \ref{Thm-Rn-Rn-Conection} is the following: 
\begin{coro} \label{Coro-tttt}
	The coefficients in the connection formula \eqref{Eq-Connection-Formula-Rn-1} satisfy
	\[
	\gamma_n = \frac{1}{2} \left[[1-ic_{n+1}(\mu)] \frac{[1-ic_{n+2}(\mu)]+ a_1^{(n)}}{[1-ic_{n+1}(\mu)]+ a_1^{(n-1)}} + [1+ic_{n+2}(\mu)]  \right] \gamma_{n-1}, \quad n\ge 1,
	\]
	with $\gamma_{0}=  [\overline{G_{2}(0)}]^{-1} [1+ic_{1}(\mu)][1+ic_{2}(\mu)]$, as well as 

	\[
	a_1^{(n)} = -\frac{R_{n+2}(z_1;\mu)z_{2}R_{n}(z_2;\mu) - R_{n+2}(z_2;\mu)z_{1}R_{n}(z_1;\mu)} {R_{n+1}(z_1;\mu)z_{2}R_{n}(z_2;\mu) - R_{n+1}(z_2;\mu)z_{1}R_{n}(z_1;\mu)}
	\]
	and
	\[
	a_2^{(n)} = \frac{R_{n+2}(z_1;\mu)R_{n+1}(z_2;\mu) - R_{n+2}(z_2;\mu)R_{n+1}(z_1;\mu)} {R_{n+1}(z_1;\mu)z_{2}R_{n}(z_2;\mu) - R_{n+1}(z_2;\mu)z_{1}R_{n}(z_1;\mu)}.
	\]
	
	Furthermore, the the  coefficients  $c_n(\nu)$ and  $g_{n}(\nu)$ corresponding to the measure $d \nu(z) =  z^{-1} G_{2}(z) \, d \mu(z)$ satisfy for $n \geq 1$,
\[
    c_n(\nu) = \frac{i}{2} \frac{\gamma_{n-1}}{\gamma_{n}} \left[[1-ic_{n+1}(\mu)] \frac{[1-ic_{n+2}(\mu)]+ a_1^{(n)}} {[1-ic_{n+1}(\mu)]+ a_1^{(n-1)}} - [1+ic_{n+2}(\mu)]  \right] 
\]
and 
\[
    g_n(\nu) =1- \frac{\gamma_{n-1}}{\gamma_{n}}\frac{\xi_{n+2}(\mu) + a_{1}^{(n)}\xi_{n+1}(\mu) + a_{2}^{(n)}\xi_{n}(\mu)} 
    {\xi_{n+1}(\mu) + a_{1}^{(n-1)}\xi_{n}(\mu) + a_{2}^{(n-1)}\xi_{n-1}(\mu)}.
\]
\end{coro}

 \section{Examples}\label{sec:examples}
 
 \begin{example}
 	We start with the simplest case of the normalized Lebesgue measure $\mu$ on $\T$, given by \eqref{normalizedLebesgue}, illustrating the discussion in Example~\ref{exampleLeb} for $m=1$. In this setting, $G_{2}(\zeta)=(\zeta-z_1)(1-\overline{z_1} \zeta)$,  $z_1\in \C\setminus\{0\}$, and for the admissible set $\mathcal P$ from \eqref{example1}, $p_0=p_1=1$, and $p_2(z)=z$. Hence,
 	the matrix in \eqref{defQmatrix} is
 	is
 	$$
 	\bm Q(z,w) =\begin{pmatrix}
 	K_{n+2}(z,w;\mu) & K_{n+1}(z,w;\mu) &   z K_{n}(z,w;\mu) \\
 	K_{n+2}(z_1,w;\mu) & K_{n+1}(z_1,w;\mu) &   z_1 K_{n}(z_1,w;\mu) \\
 	K_{n+2}(z_2,w;\mu) & K_{n+1}(z_2,w;\mu) &   z_2 K_{n}(z_2,w;\mu)
 	\end{pmatrix},\quad z_2=1/\overline{z_1}.
 	$$
 	But for the normalized Lebesgue measure $\mu$, the CD kernel satisfies the identities
$$
K_{n+1}(z,w)-K_n(z,w)=\overline{w}^{n+1} z^{n+1},  
\quad 
\overline{w} zK_n(z, w)=K_{n+1}(z,w)-1,
$$
which can be used to simplify the expression for $\bm Q(z,w)$ for $w\neq 0$:
\begin{align*}
\det \bm Q(z) = & \frac{1}{\overline{w}}\,\det \begin{pmatrix}
K_{n+2}(z,w;\mu) & K_{n+1}(z,w;\mu) &    K_{n+1}(z,w;\mu)-1 \\
K_{n+2}(z_1,w;\mu) & K_{n+1}(z_1,w;\mu) &     K_{n+1}(z_1,w;\mu)-1 \\
K_{n+2}(z_2,w;\mu) & K_{n+1}(z_2,w;\mu) &     K_{n+1}(z_2,w;\mu)-1
\end{pmatrix}\\
= & \frac{1}{\overline{w}}\,\det \begin{pmatrix}
\overline{w}^{n+2} z^{n+2} & K_{n+1}(z,w;\mu) &     -1 \\
\overline{w}^{n+2} z_1^{n+2} & K_{n+1}(z_1,w;\mu) &     -1 \\
\overline{w}^{n+2} z_2^{n+2} & K_{n+1}(z_2,w;\mu) &   -1
\end{pmatrix}
\\
= & \overline{w}^{n+1}\,\det \begin{pmatrix}
 z^{n+2} & K_{n+1}(z,w;\mu) &    -1 \\
  z_1^{n+2}-z^{n+2} & K_{n+1}(z_1,w;\mu)-K_{n+1}(z,w;\mu) &  0 \\
 z_2^{n+2}-z^{n+2} & K_{n+1}(z_2,w;\mu)-K_{n+1}(z,w;\mu) &   0
\end{pmatrix}\\
= & -\overline{w}^{n+1}\,\det \begin{pmatrix}
z_1^{n+2}-z^{n+2} & K_{n+1}(z_1,w;\mu)-K_{n+1}(z,w;\mu)  \\
z_2^{n+2}-z^{n+2} & K_{n+1}(z_2,w;\mu)-K_{n+1}(z,w;\mu)  
\end{pmatrix}.
\end{align*}
Recall that
$$
K_{n}(z,w;\mu)=\sum_{j=0}^n \overline{w}^j z^j = 1+ \overline{w} z +\mathcal O(w^2), \quad w\to 0,
$$
so that, as $w\to 0$,
\begin{align*}
\det \bm Q(z) = & \overline{w}^{n+2}\, \left((z_2^{n+2}-z^{n+2}) (z_1-z) - (z_1^{n+2}-z^{n+2}) (z_2-z)+\mathcal O(w)\right).
\end{align*}
Using the arguments from Example~\ref{exampleLeb}, for 
$$
 d\nu(\zeta)=\frac{1}{2\pi}|\zeta-z_1|^2 |d\zeta|, \quad \zeta\in \T,\quad z_1\in \C\setminus\{0\},
 $$ 
we have
\begin{align*}
K_n(z,0;\nu)=& \overline{\varphi_{n}^{\ast}(0; \nu)}\,\varphi_{n}^{\ast}(z; \nu) \\
=& \const \frac{ (z^{n+2}-z_1^{n+2}) (z-z_2)-(z^{n+2}-z_2^{n+2}) (z-z_1)}{(z-z_1)( z-z_2)},   \quad z_2=\frac{1}{\overline{z_1}}\neq 1.
\end{align*}
\end{example}
 
\vspace{2ex}
 
\begin{example}
   We consider the positive measure on the unit circle given by
\begin{equation} \label{Eq-qMeasure}
    d \mu^{(b)}(\zeta) =  \rho^{(b)}\, \frac{|(q\zeta;\,q)_{\infty}|^2} {|(q^{b}\zeta;\,q)_{\infty}|^2}\frac{1}{2 \pi i \, \zeta} \,d\zeta, 
\end{equation}
where $0 < q < 1$ and $\Re(b) > 0$.  Recall that
$$
(a;q)_n=\prod_{j=0}^{n-1}(1-a q^j).
$$
The choice 
\[
     \rho^{(b)} = \frac{(1-q^{\overline{b}})}{\,_2\phi_1(q,\, q^{-b+1}; q^{\overline{b}+1};q,q^{b})}\, \frac{(q;\,q)_{\infty}\,(q^{b+\overline{b}};\,q)_{\infty}} {(q^{b};\,q)_{\infty}\,(q^{\overline{b}};\,q)_{\infty}}.
\]
makes $\mu^{(b)}$ a probability measure (see \cite{SRanga-2016}). It is also known that the normalized CD-kernels, defined as in \eqref{Eq-Normalized-CDKernel}, are
\begin{equation} \label{Eq-Rk-ExplicitForm}
\begin{array}{ll}
R_{n}(z; \mu^{(b)}) & \dsp  =  \frac{(q^{\overline{b}};\,q)_n}{(q^{\lambda}\cos(\eta_q);\,q)_n}\,
\,_2\phi_1\Big(\begin{array}{c}
q^{-n},\, q^{b} \\
q^{-\overline{b}-n+1}
\end{array}\!\!;\, q,\, q^{-\overline{b}+1}z\Big), \quad  n \geq 0,
\end{array} 
\end{equation} 
where $\lambda=\Re(b)$, $\eta=-\Im(b)$, and $\eta_{q} = \eta\, \ln(q)$. Furthermore, the $(c_n,g_n)$ parametrization of $\mu^{(b)}$ (see Section~\ref{Sec-ModifiedCDKernel}) is given by

\begin{align*} 
c_{k}(\mu^{(b)}) &= \frac{q^{\lambda+k-1}\sin(\eta_q)}{1 - q^{\lambda+k-1}\cos(\eta_q)}, \\
g_{k}(\mu^{(b)}) &=\frac{1}{2}\frac{1-q^{\overline{b}+k-1}}{1-q^{\lambda+k-1}\cos(\eta_{q})} \frac{\,_2\phi_1(q^{k-1},\, q^{-b+1}; q^{\overline{b}+k-1};q,q^{b})}{\,_2\phi_1(q^{k},\, q^{-b+1}; q^{\overline{b}+k};q,q^{b})}, 
\end{align*}
with $ k \geq 1$, see \cite{SRanga-2016}.

If we consider the normalized kernel polynomials $R_{n}(.;\nu^{(b,z_1)})$ with respect to the measure  
\[
    d \nu^{(b,z_1)}(\zeta) = \frac{1}{\zeta} G_2^{(z_1)} (\zeta)\, d \mu^{(b)}(\zeta), \quad \zeta\in \T,
\]
where $G_2^{(z_1)} (\zeta) =-z_{2}^{-1}(\zeta - z_1) (\zeta - z_{2})$, with $|z_1| > 1$ and $z_2 = 1/\overline{z}_{1}$, then all formulas from Corollary~\ref{Coro-tttt} apply. 

The choice  
$$
z_1=q^{-b}, \quad  z_2 = q^{\overline{b}},
$$
is particularly interesting, since the measure $\nu^{(b,q^{-b})}$  coincides, up to a multiplicative constant, with $  \mu^{(b+1)}$, which leads to the following connection formula:
\begin{theo}\label{thm:example2}
Let $R_{n}(z; \mu^{(b)})$ be given by \eqref{Eq-Rk-ExplicitForm}. Then, for $n \geq 0$,
 \begin{equation*} \label{Eq-Connection-Formula-Rn-q}
\begin{array}{l}
\dsp \gamma_{n}^{(b)}  G_{2}^{(q^{-b})}(z) R_{n}(z;\mu^{(b+1)}) =  R_{n+2}(z; \mu^{(b)}) +  a_{1}^{(n)}\,   R_{n+1}(z; \mu^{(b)})  +  a_{2}^{(n)} z R_{n}(z; \mu^{(b)}), 
\end{array}
\end{equation*}
where
\begin{align*} 
 a_{1}^{(n)}& =   \frac{(1-q^{2\lambda+n+1})}{[1-q^{\lambda+n+1}\cos(\eta_q)]} \frac{2iq^{\lambda}\sin(\eta_{q})}{q^{b}(1-q^{\overline{b}})} , \\   a_{2}^{(n)} & = \frac{(1-q^{2\lambda+n})(1-q^{2\lambda+n+1})}{[1-q^{\lambda+n}\cos(\eta_q)][1-q^{\lambda+n+1}\cos(\eta_q)]} \frac{(1 - q^{-b})}{(1-q^{\overline{b}})},
\end{align*}
 and
 \[
 \gamma_{n}^{(b)} = -q^{\overline{b}}\frac{1-q^{b+n+1}}{1-q^{\lambda+n+1}\cos(\eta_{q})} \frac{1-q^{b}}{1-q^{\lambda}\cos(\eta_{q})}.
 \]
\end{theo}
\begin{proof}
From Heine's $q$-analogue of Gauss summation formula (see, for example, \cite[p.\,14]{GasRah-book}) we get 
\[
     R_{n}(q^{-b}; \mu^{(b,q^{-b})}) = \frac{(q^{2\lambda};q)_{n}}{(q^{\lambda}\cos(\eta_{q}); q)_{n}} q^{-nb}, \quad   R_{n}(q^{\overline{b}}; \mu^{(b,q^{-b})}) = \frac{(q^{2\lambda};q)_{n}}{(q^{\lambda}\cos(\eta_{q}); q)_{n}}, \quad n\ge 1,
\]
 and it remains to apply the formulas from Corollary~\ref{Coro-tttt} to arrive at 
\[
     \begin{array}l
        \dsp a_{1}^{(n)} = - \frac{(1-q^{2\lambda+n+1})}{[1-q^{\lambda+n+1}\cos(\eta_q)]} \frac{(1- q^{\overline{b}-b})}{(1-q^{\overline{b}})} = \frac{(1-q^{2\lambda+n+1})}{[1-q^{\lambda+n+1}\cos(\eta_q)]} \frac{2iq^{\lambda}\sin(\eta_{q})}{q^{b}(1-q^{\overline{b}})} , \\[3ex]
        \dsp a_{2}^{(n)} = \frac{(1-q^{2\lambda+n})(1-q^{2\lambda+n+1})}{[1-q^{\lambda+n}\cos(\eta_q)][1-q^{\lambda+n+1}\cos(\eta_q)]} \frac{(1 - q^{-b})}{(1-q^{\overline{b}})},
     \end{array}
\]
as stated.
%
Furthermore, observing that 
\[
     1+ic_{n}(\mu^{(b)}) = \frac{(1-q^{b+n-1})}{[1-q^{\lambda+n-1}\cos(\eta_{q})]}, \quad   1-ic_{n}(\mu^{(b)}) = \frac{(1-q^{\overline{b}+n-1})}{[1-q^{\lambda+n-1}\cos(\eta_{q})]} 
\]
we find that 
\[
    \big(1 - ic_{n+2}(\mu^{(b)})\big) + a_1^{(n)} = q^{\overline{b}-b}\frac{(1-q^{b})}{(1-q^{\overline{b}})}\frac{(1-q^{b+n+1})}{[1-q^{\lambda+n+1}\cos(\eta_{q})]} .
\]
Using this identity in the Corollary~\ref{Coro-tttt} we get 
\[
        \gamma_{n}^{(b)} = \frac{1-q^{b+n+1}}{1-q^{b+n}} \frac{1-q^{\lambda+n}\cos(\eta_{q})}{1-q^{\lambda+n+1}\cos(\eta_{q})} \, \gamma_{n-1}^{(b)}, \quad n \geq 1.
\]
Taking into account the initial condition $\gamma_{0}^{(b)} = [\overline{G_{2}^{(q^{-b})}(0)}]^{-1} \overline{R_{2}(0; \mu^{(b)})}$, we finally conclude that
\[
      \gamma_{n}^{(b)} = -q^{\overline{b}}\frac{1-q^{b+n+1}}{1-q^{\lambda+n+1}\cos(\eta_{q})} \frac{1-q^{b}}{1-q^{\lambda}\cos(\eta_{q})}, \quad n \geq 0.
\]
 \end{proof}

\end{example}

\vspace{2ex}
 
\begin{example}\label{example43}
	Our next example is the one-parametric family of Geronimus polynomials, defined by the constant Verblunsky coefficients,
	$$
	\alpha_n=\alpha, \quad n\geq 0, \quad 0<|\alpha| < 1.
	$$
	The corresponding unit orthogonality measure $ \mu$ is (see \cite{Geronimus-AMSTransl-1977} and \cite[p.\,83]{Simon-Book-p1}) 
	$$
	d \mu \left(e^{i\theta}\right)= \frac{\sqrt{\cos^2(\theta_{|\alpha|}/2)-\cos^2(\theta/2)}}{2 \pi |1+\alpha|\,\sin((\theta-\vartheta_{\alpha})/2)} \chi_{[\theta_{|\alpha|},2\pi-\theta_{|\alpha|}]}(\theta) + \Delta_\alpha \delta_{\vartheta_{\alpha}},
	$$
	where $\chi_{[\theta_{|\alpha|},2\pi-\theta_{|\alpha|}]}$ is the characteristic function of the interval $[\theta_{|\alpha|},2\pi-\theta_{|\alpha|}]$, with $\theta_{|\alpha|} = 2\arcsin(|\alpha|)\in [0,\pi)$; $\delta_{\vartheta_{\alpha}}$ is the Dirac delta at the point $\vartheta_{\alpha}\in (-\pi,\pi)$ defined by
	$$
	w_{\alpha} := e^{i\vartheta_{\alpha}} =  \frac{1+\overline{\alpha}}{1+\alpha},
	$$
	and the mass $\Delta_\alpha$ is given by
	$$
	\Delta_\alpha =  \frac{2}{|1+\alpha|^2} \max \{ \Re(\alpha) + |\alpha|^2, 0\}.
	$$

Let us consider the probability measure $\widetilde \mu$ obtained by rotating ${\mu} $ by the angle $-\vartheta_{\alpha}$, so that $w_\alpha \mapsto 1$.  That is, 
$$
d \widetilde\mu \left(e^{i\theta}\right)= \frac{\sqrt{\cos^2(\theta_{|\alpha|}/2)-\cos^2((\theta+\vartheta_{\alpha})/2)}}{2 \pi |1+\alpha|\,\sin(\theta/2)} \chi_{[\theta_{|\alpha|}-\vartheta_{\alpha}, 2\pi-\theta_{|\alpha|}-\vartheta_{\alpha} ]}(\theta) + \Delta_\alpha \delta_{0}.
$$
The Verblunsky coefficients of $\mu$ are 
\[
     \widetilde \alpha_{n}   = w_{\alpha}^{n+1}\, \alpha,  \quad n \geq 0. 
\]
Recall that the normalized CD Kernels \eqref{Eq-Normalized-CDKernel},
$$
R_n(z;\widetilde\mu) := \xi_n(\widetilde\mu)\, K_{n}(z,1;\widetilde\mu) , \quad n \geq 0,
$$
satisfy the three term recurrence relation \eqref{Eq-TTRR-Rn1}. From \cite{MFinkelshtein-SRanga-Veronese-2016} it follows that its coefficients \eqref{cSeq}--\eqref{def:CS1} are also constant, given by 
\begin{equation} \label{GeronimusC}
c_{n}(\widetilde\mu) = c^{(\alpha)}= -\frac{\Im(\alpha)}{1 + \Re(\alpha)},  \quad 
d_{n+1}(\widetilde \mu ) = d^{(\alpha)} = \big(1-g^{(\alpha)} \big) g^{(\alpha)}, \quad n\ge 1,
\end{equation}
with 
\begin{equation} \label{GeronimusG}
g_{n}(\widetilde \mu) = g^{(\alpha)} = \frac{1- |\alpha|^2}{2[1 + \Re(\alpha)]}, \quad n \geq 1. 
\end{equation}

From the theory of difference equations we find
\[
     R_{n}(z; \widetilde \mu ) = \frac{1}{2^{n+1}f_{2}^{(\alpha)}(z)} \big[ \big(f_1^{(\alpha)}(z) + f_2^{(\alpha)}(z)\big)^{n+1} - \big(f_1^{(\alpha)}(z) - f_2^{(\alpha)}(z)\big)^{n+1} \big],
\]
for $n \geq 0$, where
\[ 
   \begin{array} l
     f_{1}^{(\alpha)}(z) = (1+ic^{(\alpha)})z +  (1-ic^{(\alpha)}),  \\[1ex]
     f_{2}^{(\alpha)}(z) = \sqrt{(f_{1}^{(\alpha)}(z))^2 - 16 d^{(\alpha)}z} = (1+ic^{(\alpha)}) \sqrt{(z - z_{1}^{(\alpha)})(z - z_{2}^{(\alpha)})},
   \end{array}
\]
with $z_1^{(\alpha)} = e^{i(\theta_{|\alpha|}-\vartheta_{\alpha})}$ and $z_2^{(\alpha)} = e^{i(2\pi - \theta_{|\alpha|}-\vartheta_{\alpha})}$. 
Notice that  $z_1^{(\alpha)} \neq z_{2}^{(\alpha)}$ and      
\[
  R_{n}(z_{j}^{(\alpha)}; \widetilde\mu )  = \frac{n+1}{2^{n}} \big(f_{1}^{(\alpha)}(z_{j}^{(\alpha)})\big)^{n} = 2^{n} (n+1) (d^{(\alpha)})^{n/2} (z_{j}^{(\alpha)})^{n/2}, \quad j = 1,2, 
\]
for $n \geq 0$. 

The conjugate reciprocal polynomial 
\[
     G_{2}^{(\alpha)}(z) = - e^{i\vartheta_{\alpha}} (z - z_{1}^{(\alpha)}) (z - z_{2}^{(\alpha)})
\]
is such that $e^{-i\theta} G_2^{(\alpha)}(e^{i\theta})\geq 0$ for $\theta \in [\theta_{|\alpha|}-\vartheta_{\alpha}, 2\pi -\theta_{|\alpha|}-\vartheta_{\alpha} ]$ (although it is  negative outside of this interval). With the assumption $\Re(\alpha) + |\alpha|^2 \leq 0$ we  consider $\nu^{(\alpha)}$ given by
\[
       d \nu(\zeta) = \frac{1}{\zeta} G_{2}^{(\alpha)}(\zeta) d\widetilde  \mu(\zeta),
\]
which is again a positive measure on $\T$, and we can calculate the coefficients in the relation \eqref{Eq-Connection-Formula-Rn-1}.

It turns out  (see formulas in Corollary \ref{Coro-tttt}) that
\begin{align*} 
   a_1^{(n )}  &=   -\frac{R_{n+2}(z_1^{(\alpha)}; \widetilde \mu )z_{2}^{(\alpha)}R_{n}(z_2^{(\alpha)}; \widetilde \mu ) - R_{n+2}(z_2^{(\alpha)}; \widetilde \mu )z_{1}^{(\alpha)}R_{n}(z_1^{(\alpha)}; \widetilde \mu )}{R_{n+1}(z_1^{(\alpha)}; \widetilde \mu )z_{2}^{(\alpha)}R_{n}(z_2^{(\alpha)}; \widetilde \mu ) - R_{n+1}(z_2^{(\alpha)}; \widetilde \mu )z_{1}^{(\alpha)}R_{n}(z_1^{(\alpha)}; \widetilde \mu )} = 0
 \end{align*}
and
\begin{align*} 
   a_2^{(n )}   &=  \frac{R_{n+2}(z_1^{(\alpha)}; \widetilde \mu)R_{n+1}(z_2^{(\alpha)}; \widetilde \mu) - R_{n+2}(z_2^{(\alpha)}; \widetilde \mu)R_{n+1}(z_1^{(\alpha)}; \widetilde \mu)}{R_{n+1}(z_1^{(\alpha)}; \widetilde \mu)z_{2}^{(\alpha)}R_{n}(z_2^{(\alpha)}; \widetilde \mu) - R_{n+1}(z_2^{(\alpha)}; \widetilde \mu)z_{1}^{(\alpha)}R_{n}(z_1^{(\alpha)}); \widetilde \mu)} \\ 
   &=  -4\frac{n+3}{n+1}d^{(\alpha)}.
 \end{align*}
Thus,  \eqref{Eq-Connection-Formula-Rn-1} takes the form  
\begin{equation*} 
  \begin{array}{l}
    \dsp {\gamma_{n}^{(\alpha)}}  G_{2}^{(\alpha)}(z) R_{n}(z; \nu) =  R_{n+2}(z; \widetilde \mu) -   4\frac{n+3}{n+1}d^{(\alpha)}  z R_{n}(z; \widetilde \mu), 
  \end{array}
\end{equation*}
with $d^{(\alpha)}$ given by \eqref{GeronimusC}--\eqref{GeronimusG}.

We can apply Corollary \ref{Coro-tttt} in order to find  the multiplication constant $\gamma_{n}^{(\alpha)}$ and  the coefficients  $c_n(\nu)$,  $g_n(\nu)$ that appear in the three term recurrence \eqref{Eq-TTRR-Rn1} for the normalized CD kernels $R_n(z; \nu)$: now 
\[
      \gamma_{n-1}^{(\alpha)}= \frac{|1+\alpha|^2}{[1 + \Re(\alpha)]^2}, \quad c_{n}(\nu) = -\frac{\Im(\alpha)}{1 + \Re(\alpha)},
\]  
for $n \geq 1$, and 
\[
      g_{n}(\nu) = 1-\frac{n}{n+1} \frac{(n+1)(1-g^{(\alpha)}) - 4(n+3)}{n(1-g^{(\alpha)}) - 4 (n+2)} (1-g^{(\alpha)}), \quad n \geq 1.    
\]
\end{example}

\vspace{2ex}

\begin{example}\label{example44} 
As our last example we consider the probability measure on the unit circle given by $d \mu(e^{i\theta}) = w^{(b)}(\theta) d\theta$, where 
\[
     w^{(b)}(\theta) = \frac{2^{b+\overline{b}}|\Gamma(b+1)|^2 }{2\pi\, \Gamma(b+\overline{b}+1)} e^{(\pi-\theta)\Im(b)} [\sin^{2}(\theta/2)]^{\Re(b)}.
\]
Here, $b = \lambda+ i\eta$ and $\lambda > -1/2$. From \cite{SR-PAMS2010} we know that the associated monic orthogonal polynomials $\Phi_{n}(z; \mu)$ and the normalized CD kernels $R_{n}(z;\mu) =  \xi_{n}(\mu) K_{n}(z,1;\mu)$ are, respectively,  
\[
    \Phi_{n}(z; \mu) = \Phi_{n}^{(b)}(z) =  
           \frac{(2 \lambda+1)_{n}}{(b+1)_{n}}\, _2F_1(-n,b+1;\,b+\overline{b}+1;\,1-z),   \quad   n \geq 0.
\]
and 
\begin{equation} \label{Eq-Rk41-ExplicitForm}
   R_{n}(z; \mu) = R_{n}^{(b)}(z)   = \frac{(2\lambda+2)_n}{(\lambda+1)_n} \, _2F_1(-n,b+1;\,b+\bar{b}+2;\,1-z),  \quad n \geq 0.
\end{equation}
Moreover, the $(c_n,g_n)$ parametrization of $\mu$ (see Section~\ref{Sec-ModifiedCDKernel}) is given by
\[
c_{n}(\mu) = c_{n}^{(b)} = \frac{\eta}{\lambda+n}, \quad  
g_{n}(\mu) =g_{n}^{(b)} = \frac{1}{2} \frac{2\lambda+n}{\lambda+n}, 
\]
with $n \geq 1$. 

Now, if we consider the normalized CD kernels $R_{n}(.;\nu)$ with respect to the measure 
\[
    d \nu(\zeta) = \frac{1}{\zeta} G_2^{(z_1)} (\zeta)\, d \mu(\zeta), \quad \zeta\in \T,
\]
where $G_2^{(z_1)} (z) =-z_{2}^{-1}(\zeta - z_1) (\zeta - z_{2})$, with $|z_1| > 1$ and $z_2 = 1/\overline{z}_{1}$, then \eqref{Eq-Connection-Formula-Rn-1} together with the formulas in Corollary~\ref{Coro-tttt} hold. 

However, if $|z_1| = 1$, so that $z_1 = z_2$,   then  in Corollary~\ref{Coro-tttt} the expressions for the evaluations of $a_1^{(n)}$ and $a_{2}^{(n)}$ have to be replaced by (see Remark~\ref{Rmk-1.1}) expressions involving also the derivatives. 

Let us consider $z_1 = z_2 = 1$.  In this case, $d \nu(e^{i \theta})$ coincides, up to a multiplicative constant, with $w^{(b+1)}(\theta)d \theta$ and 
\begin{equation} \label{Eq-Rk42-ExplicitForm}
   R_{n}(z; \nu) = R_{n}^{(b+1)}(z),  \quad n \geq 1.
\end{equation}
Hence, if we write \eqref{Eq-Connection-Formula-Rn-1} in the form 
\[
\begin{array}{l}
\dsp \gamma_{n}^{(b)} (z-1)^2 R_{n}^{(b+1)}(z) =  R_{n+2}^{(b)}(z) +  a_{1}^{(n,b)}\,   R_{n+1}^{(b)}(z)  +  a_{2}^{(n,b)} z R_{n}^{(b)}(z),  \quad n \geq 0, 
\end{array}
\]
we obtain 
\[ 
  \begin{array}{lll}
    a_{1}^{(n,b)} R_{n+1}^{(b)}(1)&+ \ a_{2}^{(n,b)} R_{n}^{(b)}(1) & = - R_{n+2}^{(b)}(1), \\[1ex]
    a_{1}^{(n,b)} R_{n+1}^{(b)\prime}(1)&+ \ a_{2}^{(n,b)} [R_{n}^{(b)}(1) + R_{n}^{(b)\prime}(1)] & = - R_{n+2}^{(b)\prime}(1).
  \end{array}
\]
Since
\[
    R_{n}^{(b)}(1) = \frac{(2\lambda + 2)_{n}}{(\lambda + 1)_{n}}, \quad R_{n}^{(b)\prime}(1) = \frac{n}{2} \frac{(2\lambda + 2)_{n}}{(\lambda + 1)_{n}} \frac{b+1}{\lambda+1}, \quad n \geq 0,
\]
one easily finds
\[
 a_{1}^{(n,b)} =  \frac{b-\overline{b}}{\overline{b}+1} \, \frac{2\lambda+n+3}{\lambda+n+2} ,  \quad 
 a_{2}^{(n,b)} = - \frac{b+1}{\overline{b}+1}\,  \frac{(2\lambda+n+2)(2\lambda+n+3)}{(\lambda+n+1)(\lambda+n+2)} ,
\]
for $n \geq 0$. Moreover, since 
\[
    [1 - i c_{n+2}(\mu)] + a_{1}^{(n)} = [1 - i c_{n+2}^{(b)}] + a_{1}^{(n,b)}= \frac{b+1}{\overline{b}+1} \, \frac{b+n+2}{\lambda+n+2}, 
\]
it follows from  Corollary~\ref{Coro-tttt} that
\[
 \gamma_{n}^{(b)} = \frac{b+1}{\lambda+1}\, \frac{b+n+2}{\lambda+n+2}, \quad n \geq 0.
\]

 \end{example} 

\section*{Acknowledgements}

The research of first author (CFB) was partially supported by grant 305208/2015-2 of CNPq and research project 2014/22571-2 of FAPESP of Brazil.

The second author (AMF) was partially supported by the Spanish Government together with the European Regional Development Fund (ERDF) under grant \linebreak MTM2014-53963-P from MINECO, by Junta de Andaluc\'{\i}a (the Excellence Grant P11-FQM-7276 and the research group FQM-229), and by Campus de Excelencia Internacional del Mar (CEIMAR) of the University of Almer\'{\i}a.

The research of the third author (ASR) was partially supported by grants \linebreak 305073/2014-1 and 475502/2013-2 of CNPq and the research project 2016/13309-8 of FAPESP of Brazil. 

Part of this work was carried out during the visit of AMF to the Department of Applied Mathematics of  IBILCE, UNESP. He acknowledges the hospitality of the hosting department, as well as a the financial support of the Special Visiting Researcher Fellowship 401891/2013-5 of the Brazilian Mobility Program ``Science without borders''. 

The manuscript was completed while ASR and AMF were visiting the Shanghai Jiao Tong University (SJTU) in the fall of 2016, AMF as a Visiting Chair Professor. They both would like to thank the SJTU for providing them with excellent working environment. ASR would also like to thank Mikhail Tyaglov of Shanghai Jiao Tong university for the kind invitation.



\end{document}